\documentclass[12pt]{amsart}
\usepackage[T1]{fontenc}
\usepackage[utf8]{inputenc}

\usepackage{amsmath,amsfonts,amsthm,amssymb}
\usepackage{hyperref}
\usepackage{tikz}
\usepackage{lscape}
\usepackage{xcolor}
\usepackage{caption}
\usepackage{subcaption}

\renewcommand{\arraystretch}{1.3}

\usepackage{a4wide}
\usepackage{enumerate}
\usepackage{multicol, multirow, longtable}

\newtheorem{theorem}{Theorem}[section]
\newtheorem{lemma}[theorem]{Lemma}
\newtheorem{proposition}[theorem]{Proposition}

\newtheorem{theo}{Theorem}

\theoremstyle{remark}

\newtheorem{remark}[theorem]{Remark}

\newtheorem*{claim*}{Claim}

\theoremstyle{definition}

\numberwithin{equation}{section}

\newcommand{\C}{\ensuremath{\mathbb{C}}}
\newcommand{\R}{\ensuremath{\mathbb{R}}}

\newcommand{\g}[1]{\ensuremath{\mathfrak{#1}}}

\DeclareMathOperator{\tr}{tr}

\DeclareMathOperator{\Ad}{Ad}

\DeclareMathOperator{\spann}{span}
\DeclareMathOperator{\diag}{diag}
\DeclareMathOperator{\Ric}{Ric}

\newcommand{\su}[1]{\ensuremath{\mathsf{SU}_{#1}}}
\renewcommand{\u}[1]{\ensuremath{\mathsf{U}_{#1}}}

\renewcommand{\sp}[1]{\ensuremath{\mathsf{Sp}_{#1}}}

\newcommand{\lieG}{{\mathfrak{g}}}
\newcommand{\lieH}{{\mathfrak{h}}}

\newcommand{\lieP}{{\mathfrak{p}}}

\makeatletter
\newsavebox{\@brx}
\newcommand{\llangle}[1][]{\savebox{\@brx}{\(\m@th{#1\langle}\)}
	\mathopen{\copy\@brx\kern-0.5\wd\@brx\usebox{\@brx}}}
\newcommand{\rrangle}[1][]{\savebox{\@brx}{\(\m@th{#1\rangle}\)}
	\mathclose{\copy\@brx\kern-0.5\wd\@brx\usebox{\@brx}}}
\makeatother

\usepackage{parskip}

\begin{document}

\title[Ricci flow in dimension seven and thirteen]{Positive sectional curvature is not preserved under the Ricci flow in dimensions seven and thirteen}

\author[D.~Gonz\'alez-\'Alvaro]{David Gonz\'alez-\'Alvaro}
\address{Universidad Polit\'ecnica de Madrid, Spain}
\email{david.gonzalez.alvaro@upm.es}

\author[M.~Zarei]{Masoumeh Zarei}
\address{Mathematisches Institut, Universit\"at M\"unster, 
 M\"unster, Germany}
\curraddr{Fachbereich Mathematik, Universität Hamburg, Hamburg, Germany}
\email{masoumeh.zarei@uni-hamburg.de}

\begin{abstract}
We prove that there exist $\su{3}$-invariant metrics on Aloff--Wallach spaces $W^7_{k_1, k_2}$, as well as $\su{5}$-invariant metrics on the Berger space $B^{13}$, which have positive sectional curvature and evolve under the Ricci flow to metrics with non-positively curved $2$-planes. 
\end{abstract}


\maketitle

\section{Introduction}

The Ricci flow of a Riemannian manifold $(M,g)$ is a one-parameter family of metrics $g(\ell)$ on $M$ satisfying the evolution equation
$$
\frac{\partial}{\partial \ell} g(\ell)=-2\Ric(g(\ell))
$$
with $g(0)=g$. Since its introduction by Hamilton in 1982 \cite{Ham82}, it has been a fundamental question to understand which curvature conditions are preserved under the flow, see \cite{Ni14,Ri14}. 

While positive scalar curvature and positive curvature operator are preserved in all dimensions \cite{Ham82,Ham86, Wi13, BW08}, there exist infinitely many dimensions where certain curvature conditions lying in between are not preserved \cite{AN16,GZ24}.

In this article we are interested in manifolds with positive sectional curvature, denoted by $\sec>0$. Hamilton showed in \cite{Ham82} that $\sec>0$ is preserved in dimensions $2$ and $3$. In contrast, there exist manifolds in dimensions $4,6,12$ and $24$ with metrics of $\sec>0$ which lose this property when evolved under the Ricci flow, see \cite{BeK23,BW07,CW15}. Our main results add dimensions $7$ and $13$ to this list, and provide  the first odd dimensional examples of this phenomenon.
 
\begin{theo}\label{THM:main_thm_AW}
On infinitely many Aloff--Wallach spaces $W_{k_1, k_2}^7$, there exist $\su{3}$-invariant metrics of positive sectional curvature which evolve under the Ricci flow to metrics with non-positively curved $2$-planes.
\end{theo}

In fact, we show  that the conclusion of Theorem~\ref{THM:main_thm_AW} holds for all (positively curved) Aloff--Wallach spaces. However, for the latter we can only give a computer-assisted proof. In contrast, the proof of Theorem~\ref{THM:main_thm_AW} is theoretical.

In the particular case of the Aloff--Wallach space $W_{1,1}^7$ we can moreover show that there are $\sec>0$ metrics whose Ricci flow leaves the set of $\sec\geq 0$ metrics. The same statement holds true for the Berger space $B^{13}$.

\begin{theo}\label{THM:main_thm_B13}
There exist $\su{3}$-invariant metrics on the Aloff--Wallach space $W_{1,1}^7$, as well as $\su{5}$-invariant metrics on the Berger space $B^{13}$, which have positive sectional curvature and evolve under the Ricci flow to metrics with negatively curved $2$-planes.
\end{theo}

The strategy for proving the theorems is as follows. Since the manifolds  under consideration are homogeneous spaces $G/H$, we may restrict our attention to the set of $G$-invariant metrics on them, which remains invariant under the Ricci flow. Moreover, for some spaces, there are subfamilies of $G$-invariant metrics which are  simpler to describe, yet are Ricci flow invariant. In our cases, these metrics are  described by a finite set of positive parameters $x_1,\dots,x_n$, and the Ricci flow is determined by an ODE system $x_i^\prime=-2r_i(x_1,\dots,x_n)x_i$, where $i=1,\dots, n$ and $r_i$ denote the eigenvalues of the Ricci tensor. Next, we need a characterization of those metrics which have $\sec>0$, and of those metrics with negatively curved $2$-planes in the case of Theorem~\ref{THM:main_thm_B13}. For each of the families of metrics we are considering, this characterization has been provided by the work of P\"uttmann in \cite{Pue99}. With all this information at hand, the results follow from analyzing the behavior of a metric on the boundary of the set of $\sec>0$ metrics. Let us give more details for each of the theorems.

Aloff--Wallach spaces, denoted by $W_{k_1,k_2}^{7}$, constitute an infinite family of $7$-dimensional homogeneous spaces of the form $\su{3}/\mathsf{S}^1$ which admit an $\su{3}$-invariant  metric with positive sectional curvature if and only if $k_{1}k_{2}(k_{1}+k_{2})\neq 0$ \cite{Zi07}.  Among this family, the space $W_{1,1}^7$ and those of the form $W_{n,n+1}^7$, for sufficiently large $n$,  satisfy Theorem~\ref{THM:main_thm_AW}. More generally, in any infinite sequence $\{W_{k_1^n,k_2^n}^{7}\}_{n\in\mathbb N}$, with $\frac{k_1^n}{k_2^n}$ converging to $1$, there are infinitely many spaces satisfying Theorem~\ref{THM:main_thm_AW}. Since the cohomology group $H^4(W_{k_1,k_2}^{7},\mathbb Z)$ equals the cyclic group of order $k_{1}^2+k_{2}^2+k_1k_2$, we conclude that there exist infinitely many homotopy types satisfying Theorem~\ref{THM:main_thm_AW}. We remark that, with our strategy, a theoretical proof for all cases seems impractical due to the complexity of the computations. In Section~\ref{S:computer_assisted} we, however, provide a computer-assisted proof to demonstrate that Theorem~\ref{THM:main_thm_AW} in fact holds for all Aloff--Wallach spaces.

Let us discuss the metrics on $W^7_{k_1, k_2}$ which are involved in the proofs. Aloff and Wallach \cite{AW75} were the first to discover, among a two-parameter family of $\su{3}$-invariant metrics on $W_{k_1,k_2}^{7}$,  metrics with positive sectional curvature, provided that $k_{1}k_{2}(k_{1}+k_{2})\neq 0$. Later, P\"uttmann \cite{Pue99} considered a more general four-parameter family of $\su{3}$-invariant metrics, denoted by $\langle\cdot,\cdot\rangle_{t,\vec s}$, where $\vec s=(s_0,s_1,s_2)$, and determined the subset $\Omega(k_1,k_2)$ consisting of those metrics which have $\sec>0$. It is noteworthy to mention here that since $t$ will be used as one of the parameters describing the metrics $\langle\cdot,\cdot\rangle_{t,\vec s}$ throughout this article, we  denote the time parameter in the Ricci flow equation by~$\ell$.

An important feature of the family $\langle\cdot,\cdot\rangle_{t,\vec s}$ is that it remains invariant under the Ricci flow. Moreover, the Ricci flow of a metric $\langle\cdot,\cdot\rangle_{t,\vec s}$ reduces to an ODE system in $4$ variables. In order to find metrics of $\sec>0$ which lose this property under the flow, it is essential to understand the boundary of $\Omega(k_1,k_2)$ and analyze the ODE system at the boundary points. Unfortunately, it appears that describing the boundary of $\Omega(k_1,k_2)$ and, consequently, analyzing the ODE system of a boundary point by hand in this generality is too complicated.

In order to establish a manageable setting, the crucial observation is the existence of a two-parameter and a three-parameter subfamilies of metrics on $W_{1,1}^{7}$ which remain invariant under the flow. This allows us to restrict our exploration within these two families with a twofold advantage.  First, the descriptions of the boundaries of the subsets of $\sec >0$ metrics among the two-parameter and three-parameter families are notably simpler; in fact for the two-parameter family we know the subset of metrics having negatively curved $2$-planes. Second,  the Ricci flow reduces to an ODE system in just two and three variables, respectively. This makes the analysis of the Ricci flow at the boundary points more feasible.  Then, among the two-parameter family we find $\sec>0$ metrics which evolve to metrics with negatively curved $2$-planes; this proves Theorem~\ref{THM:main_thm_B13} for $W_{1,1}^{7}$. For the three-parameter family, we detect $\sec>0$ metrics which evolve to metrics with non-positively curved $2$-planes.

With this finding concerning the three-parameter family of metrics on $W_{1,1}^{7}$ at hand, the proof of Theorem~\ref{THM:main_thm_AW} now follows from a convergence argument, which works for Aloff--Wallach spaces ``sufficiently close'' to $W^7_{1,1}$. In more detail, we introduce a variable $\xi\in (0,1]$ which  parametrizes all Aloff--Wallach spaces in a continuous way, where $W^7_{k_1,k_2}$ corresponds to $\xi=\frac{k_1}{k_2}$. This allows us to treat different spaces in a continuous manner and hence
relate the Ricci flow behavior of metrics on spaces $W^7_{k_1^n,k_2^n}$ to that of the three-parameter family on $W_{1,1}^k$, provided that $\frac{k^n_1}{k^n_2}$ tends to  $1$.

The proof of Theorem~\ref{THM:main_thm_B13} for $B^{13}$ is straightforward. The space of $\su{5}$-invariant metrics on $B^{13}=\su{5}/\sp{2}\mathsf{S}^1$ only depends on two parameters and the Ricci flow reduces to an ODE system in two variables. Moreover, the subspace of positively curved metrics as well as the subspace of metrics with negatively curved $2$-planes,  determined by P\"uttmann, have a simple description. A careful analysis of the ODE system at a boundary metric then yields the existence of $\sec>0$ metrics which evolve under the Ricci flow to metrics with negatively curved $2$-planes. 

The article is organized  as follows. In Section~\ref{S:prelim} we collect the relevant information about Aloff--Wallach spaces, and describe the aforementioned two-parameter and three-parameter families of metrics on $W_{1, 1}^7$. In Section~\ref{S:3_param_nearby_spaces} we show the existence of metrics with $\sec>0$ in the two-parameter and three-parameter families which evolve under the Ricci flow to metrics with negatively or non-positively curved $2$-planes, respectively; in particular we prove Theorem~\ref{THM:main_thm_B13} for $W_{1,1}^7$. We  illustrate this phenomenon by drawing  their phase portrait. Then we prove Theorem~\ref{THM:main_thm_AW}, and give a computer-assisted proof showing that the conclusion of Theorem~\ref{THM:main_thm_AW} holds for all Aloff--Wallach spaces. Finally, in Section~\ref{SEC:Berger} we prove Theorem~\ref{THM:main_thm_B13} for $B^{13}$. 

\subsection*{Acknowledgements} 
We would like to thank Christoph B\"ohm for his suggestion on how to extend the argument from $W^7_{1,1 }$ to infinitely many spaces. We would also like to thank Anusha Krishnan for helpful discussions. We are greatful to an anonymous referee for invaluable comments and suggestions. This project was initiated while the first author was visiting the University of M\"unster. The first author wishes to thank the University of M\"unster for providing excellent working conditions.

The first author was supported by grants PID2021-124195NB-C31 and PID2021-124195NB-C32 from the Agencia Estatal de Investigación and the Ministerio de Ciencia e Innovación (Spain). The second author was supported  by Deutsche Forschungsgemeinschaft (DFG, German Research Foundation) under Germany's Excellence Strategy EXC 2044-390685587, Mathematics M\"unster: Dynamics-Geometry-Structure, and by  DFG grant ZA976/1-1 within the Priority Program SPP2026 ``Geometry at Infinity''.

\section{Aloff--Wallach spaces}\label{S:prelim}
In this section we give the necessary background on Aloff--Wallach spaces  and describe the metrics which will be used in the proofs of the main theorems.

\subsection{Diagonal Metrics on Aloff--Wallach spaces}\label{SEC:generalities_AW}
In this subsection, we follow the notation of P\"uttmann in \cite{Pue99} to establish the space of metrics on Aloff--Wallach spaces that are of interest to us. Let 
$$\mathsf{T}^{2}=\{\diag(e^{i\theta_{0}}, e^{i\theta_{1}}, e^{i\theta_{2}}) \mid \theta_{0}+\theta_{1}+\theta_{2}=0\}$$ 
be a maximal torus of $\su{3}$. For each vector $\vec k=(k_0,k_1,k_2)\in\mathbb{Z}^3\setminus \{0\}$ with  $0\leq k_1\leq k_2$, $\gcd(k_1,k_2)=1$, and $k_0+k_1+k_2=0$, there exists an embedding of $\mathsf{S}^1$ in the maximal torus $\mathsf{T}^2<\su{3}$,  denoted by $\mathsf{S}^1_{\vec k}$,   given by $z\mapsto \diag(z^{k_0},z^{k_1},z^{k_2})$. The homogeneous space $\su{3}/\mathsf{S}^1_{\vec k}$ is called an Aloff--Wallach space and is denoted by $W_{k_1,k_2}^7$. Since $W_{0,1}^7$ does not admit positively curved homogeneous metrics, throughout the paper we will further assume that $0<k_1\leq k_2$.

Before describing the desired set of metrics on $W_{k_1,k_2}^7$, let us briefly recall some facts about invariant metrics on homogeneous spaces, see \cite[Chapter~7]{Be87} for details.  Let $G$ be a compact Lie group and $M=G/H$ be a homogeneous space. Fix a bi-invariant metric $\llangle \cdot,\cdot\rrangle$ on $G$, and let $\lieG=\lieH\oplus\lieP$ be a $\llangle \cdot,\cdot\rrangle$-orthogonal decomposition of the Lie algebra $\lieG$ of $G$, where $\lieH$ denotes the Lie algebra of $H$. Then $\lieP$ is isomorphic to the tangent space of $M$ at the origin $o=eH$. It is well-known that there is a one-to-one correspondence between $\Ad(H)$-invariant inner products on $\lieP$ and $G$-invariant metrics on $M$. Moreover, the Ricci tensor of such a $G$-invariant metric is determined by its value at the point  $o$, which gives an $\Ad(H)$-invariant symmetric bilinear form on $\lieP$.

Returning to the case of Aloff--Wallach spaces, let $t,s_0,s_1,s_2$ be positive real numbers, and denote  $\vec s=(s_0,s_1,s_2)$. Define the vector $\hat k=(\hat k_0,\hat k_1,\hat k_2)$ as $\hat k=\frac{\vec k}{|\vec k|}$, and let $\Gamma = k_1^2 + k_2^2 + k_1k_2$. Note that $|\vec k|=\sqrt{2\Gamma}$.
The tangent space of $W_{k_1,k_2}^7$ at the origin  $o$ can be described as the image of the sum of one copy $\bar\varepsilon$ of $\R$ and three copies $\C_i$ of $\C$ via the following maps:
$$
\begin{array}{cccc}
\Psi_{t,\vec s}: &\bar\varepsilon\oplus\C_0\oplus\C_1\oplus\C_2 &\to &  \g{su}_3 \\
&(x,z_0,z_1,z_2) &\mapsto& \begin{pmatrix}
\frac{2i(\hat k_1-\hat k_2)}{\sqrt{6t}}x & -\frac{z_2}{\sqrt{s_2}} &\frac{\bar z_1}{\sqrt{s_1}}\\
\frac{\bar z_2}{\sqrt{s_2}} & \frac{2i(\hat k_2-\hat k_0)}{\sqrt{6t}}x & -\frac{z_0}{\sqrt{s_0}}\\
-\frac{z_1}{\sqrt{s_1}} & \frac{\bar z_0}{\sqrt{s_0}} & \frac{2i(\hat k_0-\hat k_1)}{\sqrt{6t}}x\\
\end{pmatrix}.
\end{array}
$$

For each $(t,\vec s)$, P\"uttmann defines an $\su{3}$-invariant metric $\langle\cdot,\cdot\rangle_{t,\vec s}$ on $W_{k_1,k_2}^7$ as follows.  Denote by $\langle\cdot,\cdot\rangle$ the standard inner product on $\bar\varepsilon\oplus\C_0\oplus\C_1\oplus\C_2$. Then the metric $\langle\cdot,\cdot\rangle_{t,\vec s}$ is defined by the rule
$$\Psi_{t,\vec s}^\ast\left( \langle\cdot,\cdot\rangle_{t,\vec s}\right)=\langle\cdot,\cdot\rangle.$$
It will be convenient to define $\langle\cdot,\cdot\rangle_{t,\vec s}$ from a different point of view. To do so, note that the image of each of the summands $\bar\varepsilon,\C_0,\C_1,\C_2$ by $\Psi_{t,\vec s}$ does not depend on $t,\vec s$. Thus we may define
$$
\lieP_0=\Psi_{t,\vec s}(\bar\varepsilon),\qquad \lieP_1=\Psi_{t,\vec s}(\C_0),\qquad \lieP_2=\Psi_{t,\vec s}(\C_1),\qquad \lieP_3=\Psi_{t,\vec s}(\C_2).
$$
Observe that
$$
\lieP_0=\spann \{E\},\quad \lieP_1=\spann\{E_{23}, F_{23}\},\quad \lieP_2=\spann\{E_{13}, F_{13}\},\quad  \lieP_3=\spann\{E_{12}, F_{12}\}.
$$ 
Here $E=\diag((k_1-k_2)i,(2k_2+k_1)i,-(2k_1+k_2)i)$ and $E_{jk}$, $F_{jk}$ are defined to be the $3\times 3$ matrices with entries given by
$$
E_{jk}=\begin{cases}
1 & \text{in the position } (j,k),\\
-1 & \text{in the position } (k,j),\\
0 & \text{otherwise},\\
\end{cases}\qquad
F_{jk}=\begin{cases}
i &\text{in the position } (j,k),\\
i & \text{in the position } (k,j),\\
0 & \text{otherwise}.\\
\end{cases}
$$
Note that $\lieP_0$ and $E$ depend on $k_1,k_2$.

Let $\llangle\cdot,\cdot\rrangle$ be the bi-invariant metric on $\su{3}$ given by $\llangle X,Y\rrangle =-\frac{1}{2}\text{Re}\tr (XY)=-\frac{1}{2}\tr (XY)$, for $X,Y\in\g{su}_3$. Define $\lieH\subset\g{su}_{3}$ as the Lie algebra of $\mathsf{S}^1_{\vec k}$, i.e. $\lieH$ is spanned by $\diag(-(k_1+k_2)i,k_1i,k_2i)$. Then the orthogonal complement $\lieH^\perp$ of $\lieH$ with respect to $\llangle\cdot,\cdot\rrangle$ equals $\lieP_0\oplus\lieP_1\oplus\lieP_2\oplus\lieP_3$. Moreover,
$$
\lieP=\lieP_0\oplus\lieP_1\oplus\lieP_2\oplus\lieP_3
$$
is a $\llangle\cdot,\cdot\rrangle$-orthogonal $\Ad(\mathsf{S}^1_{\vec k})$-invariant irreducible decomposition of $\lieH^\perp$. As explained in \cite[p.~648]{Pue99}, the metrics $\langle\cdot,\cdot\rangle_{t,\vec s}$ can be written in the following diagonal form:
\begin{equation}\label{Eq:T2-Metrics}
\langle \cdot, \cdot\rangle_{t,\vec s}=t  \llangle\cdot,\cdot\rrangle|_{\lieP_0}+ s_0\llangle\cdot,\cdot\rrangle|_{\lieP_1}+ s_1\llangle\cdot,\cdot\rrangle|_{\lieP_2}+ s_2\llangle\cdot,\cdot\rrangle|_{\lieP_3}.
\end{equation}

Note that the matrices $E,E_{23}, F_{23},E_{13}, F_{13},E_{12}, F_{12}$ defined above form a basis for $\lieP$, which is adapted to the decomposition $\lieP_0\oplus\lieP_1\oplus\lieP_2\oplus\lieP_3$ by construction. Moreover, the basis is orthogonal with respect to $\llangle\cdot,\cdot\rrangle$. In fact, all matrices except $E$, whose norm is $\sqrt{3\Gamma}$,  are unit vectors. Let $\tilde E=\frac{E}{\sqrt{3\Gamma}}$. Thus we arrive to the following $\llangle\cdot,\cdot\rrangle$-orthonormal basis for $\lieP$ adapted to $\lieP_0\oplus\lieP_1\oplus\lieP_2\oplus\lieP_3$:
\begin{equation}\label{EQ:orthon_basis}
\mathcal{B}=\{\tilde E,E_{23}, F_{23},E_{13}, F_{13},E_{12}, F_{12}\}.
\end{equation}

\begin{remark}\label{REM:T2-Metrics}
By \cite[Corollary~4.3]{Pue99}, the cone of $\su{3}$-invariant and right $\mathsf{T}^2$-invariant  metrics on $W_{k_1,k_2}^7$ consists precisely of the metrics $\langle\cdot,\cdot\rangle_{t,\vec s}$. In the case of $W_{1,1}^7$, the subcone of metrics which are additionally right $\mathsf{U}_2$-invariant consists of those with $t=s_0$ and $s_1=s_2$, i.e. the metrics $\langle\cdot,\cdot\rangle_{t,t,s,s}$. Here the right $\mathsf{U}_2$-action on $W_{1,1}^7$ descends from the right $\mathsf{U}_2$-action on $\su{3}$ given by the inclusion $A\mapsto \diag(\det A^{-1},A)$.
\end{remark}

\subsection{Eigenvalues of  the Ricci tensor }\label{SEC:diagonalizing}

Here we discuss the Ricci tensor of the metrics $\langle\cdot,\cdot\rangle_{t,\vec s}$.
If $k_1\neq k_2$, then the modules $\lieP_i$ are pairwise inequivalent as $\Ad(\mathsf{S}^1_{\vec k})$-invariant representations. This has two important consequences. First, all $\su{3}$-invariant metrics on $W^{7}_{k_1,k_2}$ are of the form $\langle\cdot,\cdot\rangle_{t,\vec s}$. Second, the Ricci tensor is diagonal, meaning that it is of the form 
\begin{equation}\label{Eq:T2-Ricci}
\Ric=tr_0 \llangle\cdot,\cdot\rrangle|_{ \lieP_0}+ s_{0}r_1\llangle\cdot,\cdot\rrangle|_{\lieP_1}+ s_{1}r_2\llangle\cdot,\cdot\rrangle|_{\lieP_2}+ s_{2}r_3\llangle\cdot,\cdot\rrangle|_{\lieP_3}
\end{equation}
for certain $r_i\in\R$ which depend on $t,\vec s$, $k_1$, $k_2$.

If $k_{1}=k_{2}=1$, then  $\lieP_2$ and $\lieP_3$ are equivalent as $\Ad(\mathsf{S}^1_{1, 1})$-invariant representations, see \cite[Lemma~4.2]{Pue99}. Thus, it is a priori not clear whether the Ricci tensor of the metrics $\langle\cdot,\cdot\rangle_{t,\vec s}$ is of the form given in Equation~\eqref{Eq:T2-Ricci}. 
In order to show that it has  in fact a diagonal form, recall from Remark~\ref{REM:T2-Metrics} that the metrics $\langle\cdot,\cdot\rangle_{t,\vec s}$ are precisely the $\su{3}$-invariant and right $\mathsf{T}^2$-invariant  metrics on $W_{1,1}^7$, and again from \cite[Lemma~4.2]{Pue99} we see that the modules $\lieP_i$, $i=0, \ldots, 3$, are inequivalent as $\Ad(\mathsf T^2)$-invariant representations. This implies by Schur's Lemma that the Ricci tensor of the metrics $\langle\cdot,\cdot\rangle_{t,\vec s}$ has  diagonal form as in Equation~\eqref{Eq:T2-Ricci} as well.

Next we compute the eigenvalues of the Ricci tensor given in Equation~\eqref{Eq:T2-Ricci} (cf. \cite[Lemma~2]{W82}).  Let us recall the procedure to compute the eigenvalues of diagonal Ricci tensors in the general case. Let $G/H$ be a compact homogeneous space and $\llangle\cdot,\cdot\rrangle$ a bi-invariant metric on $G$, and denote by $\lieH^\perp$ the orthogonal complement of $\lieH$. Then, for each $G$-invariant metric $g$ on $G/H$ there is an irreducible $\Ad(H)$-invariant decomposition $\oplus_{i=0}^m\lieP_i$ of $\lieH^\perp$ such that $g$ is of the form
$$
g=x_0\llangle\cdot,\cdot\rrangle|_{\lieP_0}+ \dots + x_m\llangle\cdot,\cdot\rrangle|_{\lieP_m},
$$
for some positive $x_i\in\R$. Suppose that the Ricci tensor of $g$ is diagonal with respect to the given decomposition $\oplus_{i=0}^m\lieP_i$, i.e. it is of the form
$$
\Ric=x_{0}r_0\llangle\cdot,\cdot\rrangle|_{\lieP_0}+ \dots + x_{m}r_m\llangle\cdot,\cdot\rrangle|_{\lieP_m},
$$
for certain $r_i\in\R$. In this case, the eigenvalues $r_i$ can be computed using the following formula (see e.g. \cite[Corollary~7.38]{Be87} or \cite[Section~1]{WZ86}):
\begin{equation}\label{eqn:Ric_eig}
 r_i = \frac{b_i}{2x_i} - \frac{1}{2d_i}\sum_{j, k=0 }^m [ijk] \frac{x_j}{x_i x_k} + \frac{1}{4d_i} \sum_{j, k =0}^m [ijk] \frac{x_i}{x_j x_k}.
\end{equation}
Here $d_i=\dim \lieP_i$ and $b_i$ is determined by $-B|_{\lieP_i}=b_i \llangle\cdot, \cdot\rrangle |_{\lieP_i}$, where $B$ denotes the Killing form of $G$. The bracket constants $[ijk]$ are defined by 
$$
[ijk] = \displaystyle\sum_{\substack{e_\alpha \in \lieP_i\\ e_\beta \in \lieP_j\\ e_\gamma \in \lieP_k}} \llangle\left[e_\alpha, 
e_\beta\right], e_\gamma\rrangle^2,
$$
where $\{e_i\}$ is a $\llangle\cdot,\cdot\rrangle$-orthonormal basis of $\lieH^\perp$ adapted to $\oplus_{i=0}^m\lieP_i$. Notice that $[ijk]$ is symmetric in all indices.

In the case of the metrics $\langle \cdot, \cdot\rangle_{t,\vec s}$ from Equation~\eqref{Eq:T2-Metrics}, we have $m=3$, $x_0=t$, $x_1=s_0$, $x_2=s_1$ and $x_3=s_2$. The Killing form of $\g{su}_n$ is given by $B(X, Y) = 2n \operatorname{tr}(XY)$. Hence for $\g{su}_3$, we have $B(X, Y) = 6\operatorname{tr}(XY) = -12\llangle X, Y\rrangle$. Therefore, $b_i = 12$ for each $i \in\{0, 1, 2, 3\}$. It remains to compute the bracket constants $[ijk]$. In order to do so, we  take the basis from \eqref{EQ:orthon_basis}. For convenience we collect the Lie brackets of the elements of the basis $\mathcal{B}$ in Table~\ref{TA:Brackets}. Notice that Table~\ref{TA:Brackets} represents $[e_i, e_j]_\lieP$, where $e_i$ is the vector in row $i$, and $e_j$ is the vector in column $j$. The notation $[e_i,e_j]_\lieP$ stands for the orthogonal projection of $[e_i,e_j]$ onto $\lieP$, which provides enough information for our purposes. 

\bgroup
\def\arraystretch{1.7}
\begin{table}[!htbp]
$$
\begin{array}{c|ccccccc}
 & \tilde E & E_{23}& F_{23}&E_{13}& F_{13}&E_{12}& F_{12} \\ 
\hline
 
 \tilde E & 0 & \frac{3(k_1+k_2)}{\sqrt{3\Gamma}}F_{23} & -\frac{3(k_1+k_2)}{\sqrt{3\Gamma}}E_{23} & \frac{3k_1}{\sqrt{3\Gamma}}F_{13} & -\frac{3k_1}{\sqrt{3\Gamma}}E_{13} & -\frac{3k_2}{\sqrt{3\Gamma}}F_{12} & \frac{3k_2}{\sqrt{3\Gamma}}E_{12}\\ 
 E_{23} &&0& \frac{3(k_1+k_2)}{\sqrt{3\Gamma}}\tilde E & E_{12} & F_{12} & -E_{13} & -F_{13}\\ 
 F_{23} &&&0& -F_{12} & E_{12} & -F_{13} & E_{13} \\ 
 E_{13} &&&&0& \frac{3k_1}{\sqrt{3\Gamma}}\tilde E & E_{23} & -F_{23} \\  
 F_{13} &&&&&0& F_{23} & E_{23}\\ 
 E_{12} &&&&&&0&  - \frac{3k_2}{\sqrt{3\Gamma}}\tilde E \\  
 F_{12} &&&&&&&0
\end{array}
$$
 \caption{\label{TA:Brackets}
$\lieP$-projection of Lie brackets of the elements of the basis $\mathcal{B}$.}
\end{table}
\egroup

Using Table~\ref{TA:Brackets} we get:
\begin{proposition}\label{prop:bracketsBK}
The bracket constants for $W_{k_1,k_2}^7$ are as follows:
 \begin{align*}
 [iij] &= 0,\quad  i, j \in \{1, 2, 3\},  & [ij0] &= 0, \quad i, j \in \{1, 2, 3\},\; i \neq j,\\
 [123] &= 4, & [00i] &= 0,\quad  i\in \{1, 2, 3\},\\
 [220] &= \frac{6k_1^2}{\Gamma}, & [110] &= \frac{6(k_1+k_2)^2}{\Gamma},\\
 [330] &= \frac{6k_2^2}{\Gamma} .&& \\
 \end{align*}
\end{proposition}
Now we have all the ingredients to compute the eigenvalues of the Ricci tensor.

\begin{proposition}\label{prop:RicEigAW}
The Ricci eigenvalues of the metric $\langle \cdot, \cdot\rangle_{t,\vec s}$ on $W_{k_1,k_2}^7$ are given by
\begin{align*}
 r_0 &= \frac{3t}{2\Gamma} \left( \frac{(k_{1}+k_{2})^2}{s_0^2} + \frac{k_{1}^2}{s_1^2} + \frac{k_{2}^2}{s_2^2}\right),\\
 r_1 &= \frac{6}{s_0} -\frac{3(k_{1}+k_{2})^2t}{2\Gamma s_0^2} + \left( \frac{s_0}{s_1s_2} - \frac{s_1}{s_0s_2} - \frac{s_2}{s_0s_1}\right),\\
 r_2 &= \frac{6}{s_1} -\frac{3k_{1}^2t}{2\Gamma s_1^2} + \left( \frac{s_1}{s_0s_2} - \frac{s_0}{s_1s_2} - \frac{s_2}{s_0s_1}\right), \\
 r_3 &= \frac{6}{s_2} -\frac{3k_{2}^2 t}{2\Gamma s_2^2} + \left( \frac{s_2}{s_0s_1} - \frac{s_0}{s_1s_2} - \frac{s_1}{s_0s_2}\right).
\end{align*}
\end{proposition}

\subsection{Cone of positively curved metrics}\label{sec:parameters}
One of the main results in P\"uttmann's article is to determine the two following sets (see Remark~\ref{REM:T2-Metrics}):
\begin{itemize}
    \item the set of $\sec>0$ metrics among $\su{3}$-invariant right $\u{2}$-invariant metrics on $W_{1,1}^7$,
    \item the set of $\sec>0$ metrics among $\su{3}$-invariant right $\mathsf{T}^{2}$-invariant metrics on $W_{k_1,k_2}^7$ for any pair $k_1,k_2$.
\end{itemize}
The first set is very easy to describe. In fact, \cite[Lemma~6.1]{Pue99} gives explicitly the minimum of the sectional curvatures of $\langle\cdot,\cdot\rangle_{t,t,s,s}$, which leads to the following theorem. 

\begin{theorem}[P\"uttmann]\label{THM:cone_U2_invariant}
The metric $\langle\cdot,\cdot\rangle_{t,t,s,s}$ on $W^7_{1, 1}$ has $\sec>0$ if and only if $t<s$. Moreover, the metric $\langle\cdot,\cdot\rangle_{t,t,s,s}$ has negatively curved $2$-planes if and only if $t>s$.
\end{theorem}

Much more background is needed to give an account of the second set. Let us introduce the following notation:
$$\Omega(k_1,k_2)=\{(t,\vec s)\text{ such that } \langle\cdot,\cdot\rangle_{t,\vec s} \text{ on } W_{k_1,k_2}^7 \text{ is positively curved}\}.$$
In order to describe this set, we need to introduce a certain space $\Omega_{\hat\lambda}^\prime$ and a real valued function $t_A$ on $\Omega_{\hat\lambda}^\prime$.

For each $k_1,k_2$ and each $(t,\vec s)$, P\"uttmann defines a real function $\hat\lambda=\hat\lambda(t,\vec s,k_1,k_2)$   as the minimum of sectional curvatures of certain $2$-planes with respect to the metric $\langle\cdot,\cdot\rangle_{t,\vec s}$ on $W_{k_1,k_2}^7$, see \cite[Definition~5.1]{Pue99}. Since we only need some estimates and not the precise definition of $\hat\lambda$, we omit its technical definition. The properties of $\hat\lambda$ that we will use in this work follow from the following result of P\"uttmann:

\begin{lemma}\cite[Lemma~5.16]{Pue99}\label{L:Omega_hat}
Let $\Omega_{\hat\lambda}(k_1,k_2):=\{(t, \vec s)\mid \hat\lambda(t,\vec s,k_1,k_2)>0\}$. Then  the following hold:
\begin{enumerate}
    \item $\Omega(k_1,k_2)\subset \Omega_{\hat\lambda}(k_1,k_2)$, i.e. if $\langle\cdot,\cdot\rangle_{t,\vec s}$ on $W_{k_1,k_2}^7$ has $\sec >0$, then $\hat\lambda(t,\vec s,k_1,k_2)>0$.
    \item The condition $\hat\lambda(t,\vec s,k_1,k_2)>0$ does not depend on $t$, and hence $\Omega_{\hat\lambda}(k_1,k_2)=\mathbb{R}_+\times \Omega_{\hat\lambda}^\prime (k_1,k_2)$ for some space $\Omega_{\hat\lambda}^\prime (k_1,k_2)$.
\end{enumerate}
\end{lemma}

Part~(2) of Lemma~\ref{L:Omega_hat} holds when taking $\Omega_{\hat\lambda}^\prime(k_1,k_2)$ to be the space:
\begin{equation}\label{EQ:omega_k1k2}
    \Omega_{\hat\lambda}^\prime (k_1,k_2)= \{\vec s\text{ } \mid \text{ }\hat\lambda(t,\vec s,k_1,k_2)>0\, \,\text{for some (hence all)}\,\, t>0\}.
\end{equation}
The reader can find an explicit characterization of the set $\Omega_{\hat\lambda}^\prime (k_1,k_2)$ in Equation~\eqref{EQ:characterization_Omega_lambda} below.

The positively curved metrics on $W^{7}_{k_{1},k_{2}}$ found by Aloff and Wallach in \cite{AW75} provide an explicit subset of the set $\Omega_{\hat\lambda}^\prime (k_1,k_2)$ given in Equation~\eqref{EQ:omega_k1k2}. According to \cite{AW75}, metrics on $W^{7}_{k_{1},k_{2}}$ of the form $\langle\cdot,\cdot\rangle_{t,t,s,s}$ with $t<s$ are positively curved. Hence, as observed in \cite[p.~662]{Pue99}, it follows from Lemma~\ref{L:Omega_hat} that
\begin{equation}\label{EQ:AW}
\text{if }s_1=s_2,\text{ then }(s_0,s_1,s_2)\in\Omega_{\hat\lambda}^\prime (k_1,k_2)\text{ for all } s_0\in (0,s_1) \text{ and all }k_1,k_2.
\end{equation}
Moreover, as we shall justify, it follows from the work of P\"uttmann that $\Omega_{\hat\lambda}^\prime (k_1,k_2)$ is an open subset in $\R_+^3$. Thus triples $(s_0,s_1,s_2)$ sufficiently close to those as in Equation~\eqref{EQ:AW} also belong to $\Omega_{\hat\lambda}^\prime (k_1,k_2)$. We summarize the previous discussion in the following lemma.

\begin{lemma}\label{LEM:omega_open}
For all $k_1,k_2$, the set $\Omega_{\hat\lambda}^\prime (k_1,k_2)$ is open in $\R_+^3$ and contains the subset $\{(t,s,s)\mid t<s \}$.
\end{lemma}

We postpone the justification of the openness of $\Omega_{\hat\lambda}^\prime (k_1,k_2)$ in Lemma~\ref{LEM:omega_open} to the end of the present subsection, right below Theorem~\ref{THM:Omega}, because it uses some notions which are yet to be introduced.

Now we recall the definition of the map $t_A=t_A(\vec s,k_1,k_2)$. We need the following auxiliary functions:
\begin{align*}
\tilde a_j(\vec s) &=\frac{4}{s_j},\\
\tilde b_j(\vec s) &=-\frac{2s_1s_2 + 2s_0s_2 + 2s_0s_1-s_0^2-s_1^2-s_2^2}{s_0s_1s_2} +\frac{s_{j-1}-s_j+s_{j+1}}{s_{j-1}s_{j+1}},\\
\sigma(\vec s) &= 2s_1s_2 + 2s_0s_2 +2s_0s_1 - s_0^2 - s_1^2 - s_2^2.
\end{align*}
We define the following open subsets of $\R_{+}^{3}$:
\begin{align*}
\Omega_{\sigma}^\prime &= \{\vec s\text{ } \vert \text{ }\sigma>0\},\\
D_{\sigma}&= \Omega_{\sigma}^\prime \setminus \{(s, s, s)\mid s>0\}. 
\end{align*}
Note that by \cite[Lemma~5.17]{Pue99}, we have 
\begin{equation}\label{EQ:inclusion_Omegas}
\Omega'_{\hat\lambda}(k_1,k_2)\subset \Omega'_\sigma\quad\text{ for all } k_1,k_2.
\end{equation}
Further,  we have  the matrix $\tilde A(\vec s)$ defined as (see \cite[p.~662]{Pue99})
$$
\tilde A(\vec s)=\begin{pmatrix}
\tilde a_0 & \tilde b_2  & \tilde b_1 \\
\tilde b_2 & \tilde a_1 & \tilde b_0 \\
\tilde b_1 & \tilde b_0 & \tilde a_2 \\
\end{pmatrix}.
$$
According to \cite[Lemma~5.19]{Pue99}, the matrix $\tilde A(\vec s)$ is invertible for each $s\in D_\sigma$. 

For each vector $\vec s$ and each pair $k_1,k_2$, define the vector 
\begin{equation}\label{EQ:vector_v}
v(\vec s,k_1,k_2):=\left(\frac{\hat{k}_{0}}{s_{0}}, \frac{\hat{k}_{1}}{s_{1}}, \frac{\hat{k}_{2}}{s_{2}}\right)^{t}.
\end{equation}
Then $t_A(\vec s,k_1,k_2)$ with $\vec s\in\Omega'_\sigma$  is defined as
\begin{equation}\label{EQ:definition_tA}
t_A(\vec s,k_1,k_2)=\begin{cases}
			0, & \text{if } s_0=s_1=s_2,\\
            \frac{2}{9}\langle v,\tilde A^{-1}v \rangle^{-1}, & \text{otherwise.}
		 \end{cases}
\end{equation}
Because the entries of $\tilde A$ are rational functions, it follows that, for any fixed pair $k_1,k_2$, the map $t_A(\vec s,k_1,k_2)$ is smooth for $\vec s\in D_\sigma$.

We are ready to give the description of $\Omega(k_1,k_2)$, cf. \cite[Observation~5.23]{Pue99}:
\begin{theorem}[P\"uttmann]\label{THM:Omega}
The cone $\Omega(k_1,k_2)$ consisting of metrics $\langle\cdot,\cdot\rangle_{t, \vec s}$ on $W^{7}_{k_1,k_2}$ of $\sec>0$ can be characterized as 
$$
\Omega(k_1,k_2)=\{(t,\vec s)\in\R_+\times\Omega_{\hat\lambda}^\prime(k_1,k_2) \\\ \vert\\\ t<t_A(\vec s,k_1,k_2)\}.
$$
\end{theorem}

\begin{proof}[Proof of Lemma~\ref{LEM:omega_open}]
The fact that $\Omega_{\hat\lambda}^\prime (k_1,k_2)$ is an open subset of $\R_+^3$ can be deduced from P\"uttmann's work in several ways; we next discuss one of them. We recall the following functions from \cite[p.~652]{Pue99}:
\begin{align*}\label{EQ:cj_dj_xi_j}
    c_j&=c_j(t,\vec s,k_1,k_2)=\frac{3t\hat k_j^2}{2s_j^2},\nonumber\\
    d_j&=d_j(\vec s)=-\frac{\sigma}{4\sigma_3}+\frac{s_{j-1}-s_j+s_{j+1}}{s_{j-1}s_{j+1}},\quad \text{where } \sigma_3=s_0s_1s_2. \\
    \xi_j&=\xi_j(t,\vec s,k_1,k_2)=\sqrt{\frac{3t}{8\sigma_3}}\left( \hat k_j + \hat k_{j-1}\frac{s_j-s_{j+1}}{s_{j-1}}  +\hat k_{j+1}\frac{s_j-s_{j-1}}{s_{j+1}} \right),\nonumber\\
    \end{align*} 
and the following ones from \cite[Lemma~5.18]{Pue99}:
    \begin{align*}
    p_j&=p_j(t,\vec s,k_1,k_2)=c_{j+1}d_{j+1}+c_{j-1}d_{j-1}-(\xi_{j+1}-\xi_{j-1})^2,\\
    q_j&=q_j(t,\vec s,k_1,k_2)=4c_{j+1}d_{j+1}c_{j-1}d_{j-1}-p_j^2.
\end{align*} 
Let $\Omega'_d$ be the set $\{\vec s\text{ } \vert \text{ }d_0,d_1,d_2>0\}$. By \cite[Lemma~5.17]{Pue99}, there is a chain of inclusions
$$
\Omega'_{\hat\lambda}(k_1,k_2)\subset \Omega'_d\subset \Omega'_\sigma\quad\text{ for all } k_1,k_2,
$$
which refines the inclusion given in Equation~\eqref{EQ:inclusion_Omegas}. The set $\Omega'_d$ is open in $\R_+^3$ because it is determined by the positivity of the continuous maps $d_j:\R_+^3\to\R$, which are open conditions. 

By \cite[Lemma~5.18]{Pue99}, the subset $\Omega'_{\hat\lambda}(k_1,k_2)\subset \Omega'_d$ is characterized  as follows:
\begin{equation}\label{EQ:characterization_Omega_lambda}
    \Omega'_{\hat\lambda}(k_1,k_2)=\{\vec s\in\Omega'_d\text{ } \vert \text{ } p_j>0\text{ or } q_j>0\text{ for each } j\in\{0,1,2\} \}.
\end{equation}
Since $c_j,d_j,\xi_j$ are continuous, so are the functions $p_j,q_j$. Thus $\Omega'_{\hat\lambda}(k_1,k_2)$ is open in $\Omega'_\sigma$. It follows that $\Omega'_{\hat\lambda}(k_1,k_2)$ is itself open in $\R_+^3$.
\end{proof}

\subsection{Homogeneous Ricci flow, a special space and special families of metrics}\label{SS:special_spaces_metrics}

Let us firstly review basic properties of the Ricci flow on homogeneous spaces, see e.g. \cite{La13} for details. Given a homogeneous Riemannian manifold $(M, g)$, let  $g(\ell)$ be the solution to  the Ricci flow equation with initial metric $g$. 

Let $G$ be a Lie group acting isometrically and transitively on $(M, g)$. Since the isometry group is preserved under the Ricci flow, there is a presentation of $M$ as a homogeneous space $(M, g(\ell)) = (G/H, g (\ell))$ for all $\ell$ with the same reductive decomposition $\lieG= \lieH\oplus \lieP$.

We denote by $\langle\cdot,\cdot\rangle_{\ell} =g(\ell)(o)$ the  $\Ad(H)$-invariant inner product on $\lieP$, determined by the value of $g(\ell)$ at the origin $o$. The family $\langle\cdot,\cdot\rangle_{\ell}$ is the solution to the following ODE
\begin{equation}\label{Ric_ODE}
\frac{d}{d \ell}\langle\cdot,\cdot\rangle_{\ell}  = -2\Ric (\langle\cdot,\cdot\rangle_{\ell}), 
\end{equation}
where $\Ric(\langle\cdot,\cdot\rangle_{\ell}) = \Ric(g(\ell))(o)$. Hence, the maximal interval of time where $g(\ell)$ exists is of the form $ (L_{-}, L_{+})$ for some $-\infty\leq L_{-}<0<L_{+} \leq \infty$.

Now we restrict the discussion to the spaces $W_{k_1,k_2}^7$. The fact that the Ricci tensor of a metric $\langle\cdot,\cdot\rangle_{t,\vec s}$ is of the form given in Equation~\eqref{Eq:T2-Ricci} permits to express the Ricci flow  in Equation~\eqref{Ric_ODE} as a system of ODEs. More precisely, the Ricci flow of a given initial metric  $\langle\cdot,\cdot\rangle_{t^0,\vec s^0}$, where $(t^0,\vec s^0)=(t^0,s_0^0,s_1^0,s_2^0)$, is a family of metrics $\langle\cdot,\cdot\rangle_{t(\ell),\vec s(\ell)}$ where the parameters $(t(\ell),s_0(\ell),s_1(\ell),s_2(\ell))$ satisfy the system:
\begin{equation}\label{Eq:Ricci_Flow}
\begin{cases}
t'(\ell)=-2r_0 t(\ell),\\
s_j'(\ell)=-2r_{j+1} s_j(\ell),\quad \text{ for all } j\in\{0,1,2\},\\
(t(0),s_0(0),s_1(0),s_2(0))=(t^0,s_0^0,s_1^0,s_2^0).
\end{cases}
\end{equation}
Observe that, although it is not written explicitly, the Ricci tensor eigenvalues $r_i$ depend on $(t(\ell),s_0(\ell),s_1(\ell),s_2(\ell))$ and on $k_1,k_2$. 

To prove the existence of metrics with $\sec>0$ on Aloff--Wallach spaces $W^7_{k_1, k_2}$ evolving under the Ricci flow to metrics with non-positively curved $2$-planes, our strategy consists of understanding the boundary of $\Omega(k_1,k_2)$ and analyzing the Ricci flow starting from a metric $\langle\cdot,\cdot\rangle_{t^0,\vec s^0}$ on the boundary of $\Omega(k_1,k_2)$.  In particular, we may need to compute $t_A(s_0(\ell),s_1(\ell),s_2(\ell),k_1,k_2)$, for small $\ell$, and for that we must  compute $\tilde A^{-1}$ at $(s_0(\ell),s_1(\ell),s_2(\ell))$. We notice that the expression for $\tilde A^{-1}$ becomes unmanageable at points $(s_0,s_1,s_2)$ with pairwise distinct coordinates. Unfortunately, this issue cannot be resolved for the cases where $(k_1,k_2)\neq (1,1)$ since in general the functions $(s_0(\ell),s_1(\ell),s_2(\ell))$ are pairwise distinct for small $\ell\neq 0$ (even if two of them are equal at $\ell=0$). This can be seen by inspecting the Ricci eigenvalues from Proposition~\ref{prop:RicEigAW} and the Ricci flow equations. However, for the case of $W^7_{1,1}$, the Ricci flow of a metric $\langle\cdot,\cdot\rangle_{t^0,\vec s^0}$ with $s_1^0=s_2^0$, as further clarified below, satisfies $s_1(\ell)=s_2(\ell)$ for all times~$\ell $. This observation plays a crucial role in our computations.

In more detail, for a general metric $\langle\cdot,\cdot\rangle_{t,\vec s}$ on $W^7_{1,1}$ with $s_1=s_2$, Proposition~\ref{prop:RicEigAW} shows that $r_2=r_3$, see also Equation~\eqref{EQ:Ricci_eigen_txs} below. Together with the uniqueness of solutions to the Ricci flow, this implies that if the initial condition in System~\eqref{Eq:Ricci_Flow} satisfies  $s_1^0=s_2^0$, then $s_1(\ell)=s_2(\ell)$ for all $\ell$. Hence the three-parameter family of metrics on $W_{1,1}^7$ of the form $\langle\cdot,\cdot\rangle_{t,\vec s}$ with $s_1=s_2$ remains invariant under the Ricci flow. 

Similarly, if we consider general metrics $\langle\cdot,\cdot\rangle_{t,\vec s}$ on $W^7_{1,1}$ with $t=s_0$ and $s_1=s_2$, then $r_0=r_1$ and $r_2=r_3$, see Equation~\eqref{EQ:Ricci_eigen_ts}. Thus, if $t^0=s_0^0$ and $s_1^0=s_2^0$ in System~\eqref{Eq:Ricci_Flow}, then $t(\ell)=s_0(\ell)$ and $s_1(\ell)=s_2(\ell)$ for all $\ell$. Hence, the two-parameter family of metrics on $W_{1,1}^7$ of the form $\langle\cdot,\cdot\rangle_{t,t, s, s}$  remains invariant under the Ricci flow as well. The latter also follows from the fact that the metrics $\langle\cdot,\cdot\rangle_{t,t,s,s}$ on $W_{1,1}^7$ are precisely those which are additionally right $\u{2}$-invariant and that Ricci flow preserves the isometry group.

The upshot of considering these two families is that the corresponding subset of metrics with $\sec>0$  is manageable to describe, and in fact very simple in the case of the two-parameter family by Theorem~\ref{THM:cone_U2_invariant}. Moreover, for these two families of metrics, the Ricci flow reduces to an ODE system in two or three variables instead of the four variables in System~\eqref{Eq:Ricci_Flow}, making it easier to analyze.

In Section~\ref{S:3_param_nearby_spaces} we treat each family separately and  find metrics  on $W_{1,1}^7$, within each family  which have $\sec>0$ and lose this property when evolved by the Ricci flow. More precisely, we prove the following two theorems.

\begin{theorem}[Two-parameter family]\label{THM:2-parameter}
There exist $\su{3}$-invariant and right $\u{2}$-invariant metrics $\langle\cdot,\cdot\rangle_{t,t,s,s}$ on the Aloff--Wallach space $W_{1,1}^7$ which have positive sectional curvature and evolve under the Ricci flow to metrics with  negatively curved $2$-planes.
\end{theorem}

Theorem~\ref{THM:2-parameter} implies Theorem~\ref{THM:main_thm_B13} for $W_{1,1}^7$.

\begin{theorem}[Three-parameter family]\label{THM:3-parameter}
There exist $\su{3}$-invariant metrics $\langle\cdot,\cdot\rangle_{t,s_0,s,s}$ on the Aloff--Wallach space $W_{1,1}^7$ which have positive sectional curvature and evolve under the Ricci flow to metrics with non-positively curved $2$-planes.
\end{theorem}

\begin{remark}
   To prove Theorem~\ref{THM:main_thm_AW}, we primarily need to prove Theorem~\ref{THM:3-parameter}, and then apply a continuity argument.  However, since Theorem~\ref{THM:2-parameter} shows the non-invariance of the set of metrics with $\sec>0$ on $W^7_{1,1}$ under the Ricci flow within a smaller subset of positively curved metrics, and moreover the proof is short, we consider this case as well. 
\end{remark}

\section{Two- and three-parameter families,  nearby spaces, and the general case}\label{S:3_param_nearby_spaces}

This section is divided in five subsections. In the first two subsections we discuss the two- and three-parameter families of metrics on $W_{1,1}^7$ and give the proofs of Theorems~\ref{THM:2-parameter} and \ref{THM:3-parameter}, respectively. In the third subsection, we draw the phase portraits of the (normalized) Ricci flow of the two- and three-parameter families. In the fourth subsection we prove Theorem~\ref{THM:main_thm_AW}. Finally, we show, through a computer-assisted proof, that the results of Theorem~\ref{THM:main_thm_AW}  hold for all Aloff--Wallach spaces admitting positive curvature.

\subsection{Two-parameter family} Here we prove Theorem~\ref{THM:2-parameter} (see Remark~\ref{Rem:Einstein_Ricci} for an alternative proof). 

\begin{proof}[Proof of Theorem~\ref{THM:2-parameter}]
We consider metrics on $W_{1,1}^7$ of the form $\langle\cdot,\cdot\rangle_{t,t,s,s}$. Recall from Theorem~\ref{THM:cone_U2_invariant} that $\sec>0$ if and only if $t<s$, and has negatively curved $2$-planes if $t>s$. To prove Theorem~\ref{THM:2-parameter}, we investigate the Ricci flow of the metric $\langle\cdot,\cdot\rangle_{1,1,1,1}$.

From Proposition~\ref{prop:RicEigAW} we see that the eigenvalues of the Ricci tensor of metrics $\langle\cdot,\cdot\rangle_{t,t,s,s}$ are given by:
\begin{equation}\label{EQ:Ricci_eigen_ts}
r_0=r_1=\frac{2s^2+t^2}{t s^2},\qquad r_2=r_3=\frac{3(4s-t)}{2s^2}.
\end{equation}
Let  $g(\ell):=\langle\cdot,\cdot\rangle_{t(\ell),t(\ell),s(\ell),s(\ell)}$ be the solution to the Ricci flow with the initial metric $g(0)=\langle\cdot,\cdot\rangle_{1,1,1,1}$. 
The Ricci flow is determined by the ODE system
\begin{equation}\label{EQ:flow_2_variables}
\begin{cases}
t'(\ell)=-2r_0t(\ell),\\
s'(\ell)=-2r_2s(\ell),\\
(t(0),s(0))=(1,1).
\end{cases}
\end{equation}
This system has a unique solution in an open interval $(-L,L)$ for some $L>0$. We claim that 
\begin{equation}\label{EQ:derivative_2param}
\frac{d}{d\ell}\left(\frac{t(\ell)}{s(\ell)}\right)\Bigr|_{\ell=0}>0.
\end{equation}
This implies that there exists $\bar L>0$ with $\bar L\leq L$ such that $t(\ell)<s(\ell)$ for all $\ell\in (-\bar L,0)$ and $t(\ell)>s(\ell)$ for all $\ell\in (0,\bar L)$. Thus, any metric $g(\ell)$ with $\ell\in (-\bar L,0)$ has $\sec>0$ and evolves to metrics with negatively curved $2$-planes.

It remains to prove the claim in Equation~\eqref{EQ:derivative_2param}. For simplicity, we drop the variable $\ell$ from the equation. Differentiating and using Equations~\eqref{EQ:Ricci_eigen_ts} and \eqref{EQ:flow_2_variables} we get:
$$
\frac{d}{d\ell}\left(\frac{t}{s}\right)=\frac{t^\prime s-ts^\prime}{s^2}=\frac{-2r_0 ts +2tr_2 s}{s^2}=\frac{-4s^2-5t^2+12ts}{s^2}.
$$
At $\ell=0$ we have $t=s=1$ and hence the derivative equals $3$.
\end{proof}

\subsection{Three-parameter family}\label{SS:3_param} 

Here we consider metrics on $W_{1,1}^7$ of the form $\langle\cdot,\cdot\rangle_{t,s_0,s_1,s_2}$ with $s_1=s_2$.  We denote $s:=s_1=s_2$ for simplicity. Also, in order to avoid any confusion between $s$ and $s_0$, we set $x:=s_0$. Before giving the proof of Theorem~\ref{THM:3-parameter}, let us explain how the background given in Section~\ref{S:prelim} is adapted to this family of metrics.

First, from Proposition~\ref{prop:RicEigAW} we see that the eigenvalues of the Ricci tensor of $\langle\cdot,\cdot\rangle_{t,x,s,s}$ on $W_{1,1}^7$ are:
\begin{equation}\label{EQ:Ricci_eigen_txs}
r_0 =\frac{t(2s^2+x^2)}{x^2 s^2},\qquad r_1=\frac{4xs^2-2ts^2+x^3}{x^2s^2},\qquad r_2=r_3=\frac{12s-t-2x}{2s^2}.
\end{equation}

We also need the value of $t_A(x,s,s,1,1)$. The matrix $\tilde A(x,s,s)$ is given by
$$
\tilde{A}(x,s,s)=
\begin{pmatrix}
\frac{4}{x} &\frac{ x-3s}{s^2} &\frac{ x-3s}{s^2}\\
\frac{ x-3s}{s^2} & \frac{4}{s}& -\frac{2}{s}\\
\frac{ x-3s}{s^2}&-\frac{2}{s} & \frac{4}{s}\\
\end{pmatrix}.
$$
Its inverse, which we simply denote by $\tilde{A}^{-1}$, equals 
$$
\tilde{A}^{-1}=\frac{1}{(s-x)^2(4s-x)}
\begin{pmatrix}
s^3x & \frac{s^2x(3s-x)}{2} & \frac{s^2x(3s-x)}{2}\\
 \frac{s^2x(3s-x)}{2} & \frac{s(16s^3-9s^2x+6sx^2-x^3)}{12}& \frac{s(8s^3+9s^2x-6sx^2+x^3)}{12}\\
 \frac{s^2x(3s-x)}{2} &  \frac{s(8s^3+9s^2x-6sx^2+x^3)}{12}   & \frac{s(16s^3-9s^2x+6sx^2-x^3)}{12}  \\
\end{pmatrix}.$$
In this case $v:=v(x,s,s,1,1)=\Big(-\frac{2}{x\sqrt{6}},\frac{1}{s\sqrt{6}},\frac{1}{s\sqrt{6}}\Big)^t$. A simple computation shows that
\begin{equation}\label{EQ:A_-1_tilde_v}
\tilde{A}^{-1}v=\frac{1}{\sqrt{6}(s-x)(4s-x)}\begin{pmatrix}
sx-2s^2\\
-s^2\\
-s^2\\
\end{pmatrix}.
\end{equation}
Using Equation~\eqref{EQ:definition_tA} we arrive to the expression
\begin{equation}\label{EQ:t_A_submersion}
t_A(x,s,s,1,1)=\frac{x(4s-x)}{3s}.
\end{equation}

\begin{remark}\label{Re:AW_metrics}
Note that working with the metrics of the form  $\langle\cdot,\cdot\rangle_{t, x, s, s}$ has another advantage besides the simplicity of the computations: from Lemma~\ref{LEM:omega_open} we know that $(x,s,s)\in\Omega_{\hat\lambda}^\prime$ for any $x\in(0,s)$. 
\end{remark}

For later reference, we state the following immediate consequence of Theorem~\ref{THM:Omega}, which actually holds for all $W^{7}_{k_{1},k_{2}}$ with $k_{1}k_{2}(k_{1}+k_{2})\neq 0$.

\begin{lemma}\label{LEM:metrics_near_the_boundary}
Fix $s>0$. Then, for any $x\in (0,s)$ the following hold:
\begin{enumerate}
\item If $t\in(0,t_A(x,s,s,k_1,k_2))$, then the metric $\langle\cdot,\cdot\rangle_{t, x, s,s}$ on $W^{7}_{k_{1},k_{2}}$ is positively curved.
\item If $t\geq t_A(x,s,s,k_1,k_2)$, then the metric  $\langle\cdot,\cdot\rangle_{t, x, s,s}$ on $W^{7}_{k_{1},k_{2}}$ has non-positively curved $2$-planes.
\end{enumerate}   
\end{lemma}

Now we are ready to prove Theorem~\ref{THM:3-parameter}. 

\begin{proof}[Proof of Theorem~\ref{THM:3-parameter}]
We begin the proof by choosing a suitable initial metric on $W_{1,1}^7$  for the Ricci flow. Throughout this proof we denote $t_A(x,s,s,1,1)$ simply by $t_A(x,s,s)$. We fix $x^{0}\in(0.8,1)$, and we define the metric $g:=\langle\cdot,\cdot\rangle_{t_A(x^{0},1,1), x^0, 1, 1}$.

Let $g(\ell):=\langle\cdot,\cdot\rangle_{t(\ell),x(\ell),s(\ell),s(\ell)}$ be the solution to the Ricci flow with the initial metric $g(0)=g$. The Ricci flow is determined by the ODE system
\begin{equation}\label{EQ:flow_3_variables}
\begin{cases}
t'(\ell)=-2r_0t(\ell),\\
x'(\ell)=-2r_1x(\ell),\\
s'(\ell)=-2r_2s(\ell),\\
(t(0),x(0),s(0))=(t_A(x^0,1,1),x^0,1).
\end{cases}
\end{equation}
This system has a unique solution in an open interval $(-L,L)$ for some $L>0$.

We claim that 
\begin{equation}\label{EQ:derivative}
\frac{d}{d\ell}\left(\frac{t_A(x(\ell),s(\ell),s(\ell))}{t(\ell)}\right)\Bigr|_{\ell=0}<0.
\end{equation}
This implies that there exists $L_1>0$ with $L_1\leq L$ such that the following properties hold:
\begin{itemize}
\item[(P1)] $t_A(x(\ell),s(\ell),s(\ell))>t(\ell)$ for all $\ell\in (-L_1,0)$,
\item[(P2)] $t_A(x(\ell),s(\ell),s(\ell))<t(\ell)$ for all $\ell\in (0,L_1)$.
\end{itemize}
Now note that, because $\frac{x(0)}{s(0)}=x^{0}\in (0.8,1)$, there exists some positive $L_2\leq L$ such that:
\begin{itemize}
\item[(P3)] $\frac{x(\ell)}{s(\ell)}\in (0.8,1)$ for all $\ell\in (-L_2,L_2)$. In particular, we have $x(\ell)<s(\ell)$, for all $\ell\in (-L_2,L_2)$.
\end{itemize}
By defining $\bar L=\min\{L_1,L_2\}$, it follows that (P1), (P2) and (P3) hold for all $\ell\in (-\bar L,\bar L)$. By Lemma~\ref{LEM:metrics_near_the_boundary}, the metrics $g(\ell)$ are positively curved for $\ell\in (-\bar L,0)$ and have non-positively curved $2$-planes for $\ell\in (0,\bar L)$. In summary, any metric $g(\ell)$ with $\ell\in (-\bar L,0)$ has $\sec>0$ and evolves to metrics with non-positively curved $2$-planes.

It remains to prove that Equation~\eqref{EQ:derivative} holds. For simplicity, we drop the variable $\ell$ from the functions $t(\ell),x(\ell),s(\ell)$ in the following computations. Recall from Equation~\eqref{EQ:t_A_submersion} the value of $t_A(x,s,s)$. At an arbitrary time $\ell$ we get
\begin{equation}\label{EQ:derivative_of_f}
\frac{d}{d\ell}\left(\frac{t_A(x,s,s)}{t}\right)=\frac{d}{d\ell}\left(\frac{x(4s-x)}{3st}\right)=\frac{B}{3s^2t^2},
\end{equation}
where $B$ also depends on $\ell$ and is defined as
$$
B:=2st(2sx'+2s'x-xx')-x(4s-x)(s't+st').
$$
Substituting the derivatives $t',x',s'$ by the Ricci flow equations given by System~\eqref{EQ:flow_3_variables} we get
$$
B=2txs\left( r_0(4s-x)+r_1(2x-4s)-r_2 x\right).
$$
Replacing $r_0,r_1,r_2$ by their values from Equation~\eqref{EQ:Ricci_eigen_txs} we obtain
$$
B=\frac{t}{xs}\left( 32ts^3 -12txs^2 +8stx^2 -tx^3 +16x^2s^2 -32xs^3 -20sx^3 +6x^4 \right).
$$
Define $D(\ell):=\frac{x(\ell)s(\ell)}{t(\ell)}B(\ell)$. Then
\begin{equation}\label{EQ:derivative_f}
\frac{d}{d\ell}\left(\frac{t_A(x,s,s)}{t}\right)=\frac{D(\ell)}{3txs^3},
\end{equation}
so the sign of the derivative agrees with the sign of $D(\ell)$. Thus, in order to show the claim in Equation~\eqref{EQ:derivative}, we need to prove that $D(0)<0$. Recall that at time $\ell=0$ we have
$$
(t(0),x(0),s(0))=(t_A(x^0,1,1),x^0,1)=\left(\frac{x^0(4-x^0)}{3},x^0,1\right),
$$
which yields
$$
D(0)=\frac{x^0}{3}\left( 32-32x^0-16(x^0)^2+6(x^0)^3+(x^0)^4\right).
$$
Consider the polynomial 
$$p(z)=\frac{z}{3}\left( 32-32z-16z^2+6z^3+z^4\right).$$
The roots of $p$ can be computed with the help of a computer. They are all real and are given by:
$$
z_1\approx -7.48,\qquad z_2=-2,\qquad z_3=0,\qquad z_4\approx 0.79,\qquad z_5\approx 2.69,
$$
where $z_1,z_4,z_5$ are approximations up to two digits. It turns out that $p$ is negative in the open interval $(z_4,z_5)$. In particular, $p(x^0)<0$ because $x^0$ has been chosen in the interval $(0.8,1)$. Since $p(x^0)=D(0)$, the claim in Equation~\eqref{EQ:derivative} follows.
\end{proof}

\subsection{Phase Portrait}\label{SS:PP}
To illustrate the  evolution of metrics $\langle\cdot,\cdot\rangle_{t,x,s,s}$ on $W^7_{1, 1}$ under the Ricci flow, we provide the phase portrait of the normalized Ricci flow of such metrics in Figures~\ref{Fig:PP_2} and \ref{Fig:PP_3}, for the two-, respectively, three-parameter families. 

Recall that the normalized Ricci flow of a Riemannian manifold $(M,g(0))$ is given by the equation
$$
\frac{\partial}{\partial \ell} g(\ell)=-2\Ric(g(\ell))+\frac{2r}{n}g(\ell),\qquad \text{ with } r=r(g(\ell))=\frac{\int_M S_g du_g}{\int_M du_g},
$$
where $S_g$ denotes the scalar curvature of $g=g(\ell)$ and $du_g$ is the volume element of $g(\ell)$. Note that the solutions to the Ricci flow and to the normalized Ricci flow of a manifold are rescalings of each other, so for our purposes of examining whether $\sec>0$, it is irrelevant which one we consider. 

If $(M,g)$ is homogeneous, the scalar curvature $S_g$ is a constant function and thus $r=S_g$. The constant $S_g$ equals the trace of the Ricci tensor. In the case of the metrics $\langle\cdot,\cdot\rangle_{t,x,s,s}$ on $W^7_{1, 1}$, we have that $n=7$ and
\begin{align}\label{Eq: Scal_W_11}
    S_g=r_0+2r_1+4r_2=\frac{-2ts^2-tx^2-2x^3+24sx^2+8xs^2}{x^2s^2}.
\end{align}
Since the volume remains constant under the normalized flow, we may assume that  $tx^2s^4=1$. Whence, we have $t=x^{-2}s^{-4}$. The normalized Ricci flow system reduces then to a system with two equations as follows: 
    \begin{equation}\label{EQ:PP}
        \begin{aligned}
            {x}^{\prime}&=\frac{-40x ^{3}s ^{6}+24s ^{2}-18x ^{5}s ^{4}-2x ^{2}+48x ^{4}s ^{5}}{7x ^{3}s^{6}},\\
          s^{\prime}&=\frac{-36x ^{4}s ^{5}+16x ^{3}s ^{6}+10x ^{5}s ^{4}+5x ^{2}-4s ^{2}}{7x ^{4}s ^{5}}.
        \end{aligned}
    \end{equation}
In Figures~\ref{Fig:PP_2} and \ref{Fig:PP_3} the horizontal axis represents the variable $x$, while the vertical axis represents the variable $s$. We have three disjoint regions as follows:
\begin{align*}
   \mathcal{G}&=\left\{(x, s)\in \mathbb R^2_+\mid\, x<s,\, \,x^{-2}s^{-4}<\frac{x(4s-x)}{3s}\right\}\\
    &=\{(x, s)\in \mathbb R^2_+\mid\, x<s,\,\, 3<4x^3s^4-x^4s^3\},\\
    \mathcal{P}&=\{(x, s)\in \mathbb R^2_+\mid\,  3\geq 4x^3s^4-x^4s^3\},\\
    \mathcal{W}&=\{(x, s)\in \mathbb R^2_+\mid\, x\geq s, \, \,3<4x^3s^4-x^4s^3\}. \\
\end{align*}
The set $\mathcal G$ is depicted in \emph{light green}, $\mathcal{P}$ in \emph{pink}, and $\mathcal{W}$ in \emph{white}.
The boundary of $\mathcal G$ is given by $s=x$ and $4x^3s^4-x^4s^3=3$, represented by the \emph{blue} line and the \emph{red} curve, respectively. 

Each point $(x, s) \in \mathbb{R}^2_+$ represents the metric $g(x, s) :=\langle\cdot,\cdot\rangle_{t,x,s,s}$ with $t=x^{-2}s^{-4}$ on $W^7_{1, 1}$. These regions are distinguished by our knowledge of the sign of the sectional curvature of the corresponding metrics. Namely, if $(x, s) \in \mathcal{G}$, then $g(x, s)\in \Omega$, i.e. $g(x, s)$ has $\sec > 0$; if $(x, s) \in \mathcal{P}$, then $g(x, s)$ has $2$-planes with non-positive sectional curvature; and if $(x, s) \in \mathcal{W}$, we do not have precise information about the sign of the sectional curvature.

In Figures~\ref{Fig:PP_2}(A) and \ref{Fig:PP_3}(A) we  specify two points   
$$E^{(+)}=\left(\frac{2}{5}\sqrt[7]{\frac{125}{8}}, \sqrt[7]{\frac{125}{8}}\right),\text{ in \emph{green}},\qquad E^{(-)}=\left(2\sqrt[7]{\frac{1}{8}}, \sqrt[7]{\frac{1}{8}}\right),  \text{ in \emph{magenta}},$$
which correspond to the Einstein metrics on $W_{1, 1}^7$ given in \cite[p. 25]{N04}. Note that by Theorem~\ref{THM:cone_U2_invariant}, the metric corresponding to $E^{(+)}$ has $\sec>0$ and the metric corresponding to  $E^{(-)}$  has $2$-planes with negative curvature. 

Further, let 
\begin{align*}
    \mathcal{C}=\{(x, s)\in \R^2_+\mid\, x^3s^4=1\}.
\end{align*}
This set is illustrated in \emph{purple} (partly covered by a green curve) in Figure~\ref{Fig:PP_2}(A) and represents the two-parameter family of metrics on $W^7_{1,1}$. Note that the two Einstein metrics lie on this curve.

Moreover, in Figures~\ref{Fig:PP_2}(A) and \ref{Fig:PP_3}(A)  we illustrate two solutions to System~\eqref{EQ:PP} with initial values 
$$p_1=\left(\left(\frac{10}{11}\right)^{\frac{4}{3}}, 1.1\right)\in \mathcal{C}\cap \mathcal{G},\qquad p_2=(0.87, 1.1)\in\mathcal{G},$$
respectively, depicted as the \emph{green} curve in Figure~\ref{Fig:PP_2}(A) and as the \emph{black} curve in Figure~\ref{Fig:PP_3}(A). Let us denote the green curve by $\gamma$ and the black curve by $\beta$. Note that $p_1$ represents a metric in the two-parameter family and $p_2$ represents a metric in the three-parameter family.  We remark further that since we assume the volume to remain constant throughout the flow, the initial values that we choose here are different from those we choose for the (unnormalized) Ricci flow in the proof of Theorems~\ref{THM:2-parameter} and \ref{THM:3-parameter}. Notice that the curve $\gamma$ initiating at $p_1\in\mathcal{C}$  remains in $\mathcal{C}$ while evolving, and converges to the point representing the Einstein metric $E^{(-)}$ in positive time (see Remark~\ref{Rem:Einstein_Ricci}). The same phenomenon happens for the curve $\beta$, i.e. it converges to $E^{(-)}$ in positive time. This behavior  indicates that both $\gamma$ and $\beta$ evolve to points corresponding to metrics with negatively curved $2$-planes. 

    In Figures~\ref{Fig:PP_2}(B) and \ref{Fig:PP_3}(B) we depict the same integral curves $\gamma$ and $\beta$ as in Figures~\ref{Fig:PP_2}(A) and \ref{Fig:PP_3}(A), respectively. However, to observe their behaviors  close to the boundary of $\mathcal{G}$, we provide a close-up image. 

Recall from Theorem~\ref{THM:cone_U2_invariant} that a metric $\langle\cdot,\cdot\rangle_{x,x,s,s}$ in the two-parameter family has positive curvature if and only if $x<s$. For the metrics in this family with fixed volume, i.e. those satisfying $x^3s^4=1$, the boundary condition $x=s$  reduces to $(x, s)=(1, 1)$. In Figure~\ref{Fig:PP_2}(B), there exist  $\ell_0,  \ell_1$ with $ 0<\ell_0<\ell_1$, where $\gamma(\ell_0)=(1, 1)$ and  $\gamma(\ell_1)\in \mathcal{P}$. 

Similarly, for the curve $\beta$, there exists $\tilde\ell_0>0$ such that $q=\beta(\tilde\ell_0)\in\mathcal{P}$. This indicates that the corresponding Ricci flow metric at time $\tilde\ell_0$ has a non-positively curved $2$-plane, as proven in Theorem~\ref{THM:3-parameter}.

 \begin{figure}[!tbp]
  \subfloat[]{\includegraphics[width=0.5\textwidth]{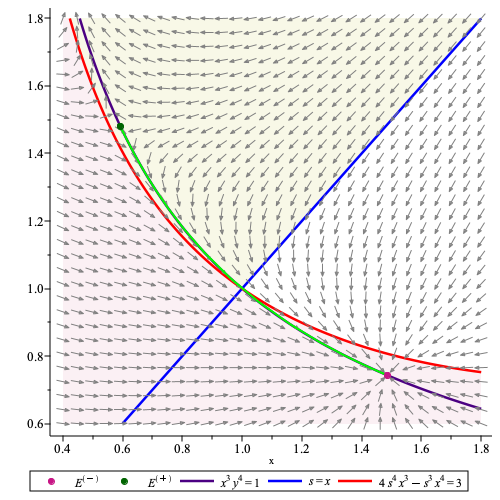}}
  \hfill
  \subfloat[]{\includegraphics[width=0.5\textwidth]{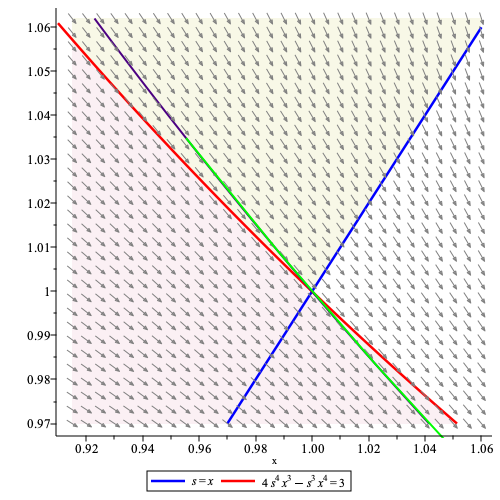}}
  \caption{Two-parameter family}
   \label{Fig:PP_2}
 \end{figure}

\begin{figure}[!tbp]
  \subfloat[]{\includegraphics[width=0.5\textwidth]{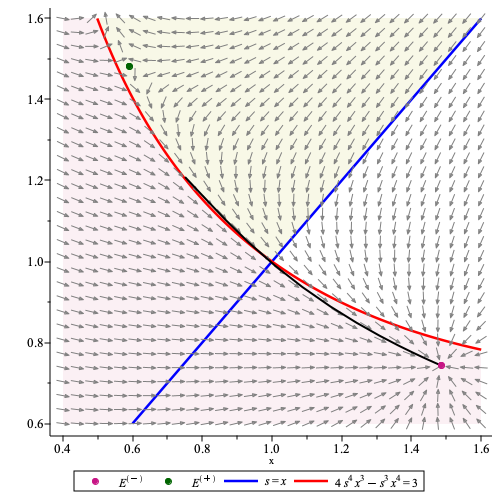}}
  \hfill
  \subfloat[]{\includegraphics[width=0.5\textwidth]{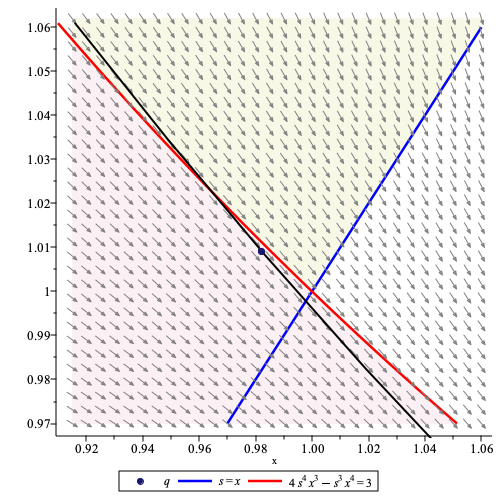}}
  \caption{Three-parameter family}
   \label{Fig:PP_3}
 \end{figure}

\subsection{Nearby spaces}\label{S:Generic_Case}

In this subsection we give the proof of  Theorem~\ref{THM:main_thm_AW}, which relies on the proof of Theorem~\ref{THM:3-parameter} and an analytic argument. To this end, we will consider infinitely many Aloff--Wallach spaces at the same time. It will be then helpful to, roughly speaking, parametrize all of them in a continuous manner. 

Recall that the Ricci eigenvalues $r_i$ from Proposition~\ref{prop:RicEigAW} depend on the parameters $k_1,k_2$, and the same holds for the vector $v$ and the map $t_A$ from Equations~\eqref{EQ:vector_v} and \eqref{EQ:definition_tA}. Clearly, the expressions for $r_i$, $v$ and $t_A$ make sense for $k_1,k_2$ not necessarily in $\mathbb N$. Moreover, observe that all of these values remain invariant under rescaling $\lambda k_1,\lambda k_2$ of $k_1,k_2$. This suggests to identify each pair $k_1,k_2\in\R_+$ with the quotient $\xi=\frac{k_1}{k_2}$. In this way, Aloff--Wallach spaces $W_{k_1,k_2}^7$ correspond to rational numbers $\xi\in (0,1]$; recall that we have assumed $\gcd (k_1, k_2)=1$ and $0<k_1\leq k_2$. Conversely, each $\xi\in (0,1]\cap\mathbb Q$ can be written as $\xi=\frac{k_1}{k_2}$ with $k_1,k_2\in\mathbb N$, $0<k_1\leq k_2$ and $\gcd(k_1,k_2)=1$, hence it corresponds to the Aloff--Wallach space $W_{k_1,k_2}^7$, which we may simply denote by $W_\xi^7$. In particular, the space $W_{1,1}^7$ corresponds to $\xi=1$.

We extend the maps $t_A$, $v$ and $r_i$ so that they depend on $\xi$ instead of $k_1,k_2$, where $\xi$ is any real number in the interval $(0,1]$. More precisely, $v:\R^3_+\times (0,1]\to\R^3$ is defined as the smooth map
$$ v(\vec s, \xi)=\left(-\frac{1+\xi}{s_0\sqrt{2(\xi^2+\xi+1)}}, \frac{\xi}{s_1\sqrt{2(\xi^2+\xi+1)}},\frac{1}{s_2\sqrt{2(\xi^2+\xi+1)}}\right).$$
The map $t_A$ extends to $t_A:D_\sigma\times \left(0,1\right]\to \R$ defined as
\begin{equation}\label{EQ:t_A-ksi}
   t_A(\vec s,\xi)=
            \frac{2}{9}\langle v(\vec s,\xi),\tilde A^{-1}(\vec s)v(\vec s,\xi) \rangle^{-1}. 
\end{equation}

Observe that $\tilde A^{-1}(\vec s)$ is a rational matrix in the variables $(s_0,s_1,s_2)$, so $t_A$ is smooth on $D_\sigma\times \left(0,1\right]$.

The Ricci eigenvalues $r_i$ extend to smooth maps $r_i:\R_+^4\times (0,1]$ as
\begin{align*}
 r_0 (t,\vec s,\xi) &= \frac{3t}{2(\xi^2+\xi+1)} \left( \frac{(\xi+1)^2}{s_0^2} + \frac{\xi^2}{s_1^2} + \frac{1}{s_2^2}\right),\\
 r_1 (t,\vec s,\xi)&= \frac{6}{s_0} -\frac{3(\xi+1)^2 t}{2(\xi^2+\xi+1) s_0^2} + \left( \frac{s_0}{s_1s_2} - \frac{s_1}{s_0s_2} - \frac{s_2}{s_0s_1}\right),\\
 r_2 (t,\vec s,\xi)&= \frac{6}{s_1} -\frac{3\xi^2t}{2(\xi^2+\xi+1) s_1^2} + \left( \frac{s_1}{s_0s_2} - \frac{s_0}{s_1s_2} - \frac{s_2}{s_0s_1}\right), \\
 r_3 (t,\vec s,\xi)&= \frac{6}{s_2} -\frac{3 t}{2(\xi^2+\xi+1) s_2^2} + \left( \frac{s_2}{s_0s_1} - \frac{s_0}{s_1s_2} - \frac{s_1}{s_0s_2}\right).
\end{align*}

For the rest of the section, we fix $x^0\in (0.8, 1)$ and let  
\begin{equation}\label{EQ:def_of_p}
p=\left(t_A\left(x^0, 1, 1, 1\right), x^0,1,1,1\right)=\left(\frac{x^0(4-x^0)}{3}, x^0, 1, 1, 1\right), 
\end{equation}
where $t_A\left(x^0, 1, 1, 1\right)$ is defined by Equation~\eqref{EQ:t_A-ksi}, with $\vec s=(x^0, 1, 1)$ and $\xi=1$.

We also define the map
\begin{align}
F\colon \R_+\times D_{\sigma}\times \textstyle{\left(0, 1\right]}&\to \mathbb{R}\label{EQ:map_F}\\
(t, s_0, s_1, s_2, \xi)&\mapsto \frac{t_A(s_0, s_1, s_2, \xi)}{t}. \nonumber
\end{align}

Denote by $g(\ell):=\langle\cdot,\cdot\rangle_{t(\ell),x(\ell),s(\ell),s(\ell)}$ the solution to the Ricci flow on $W_{1,1}^7$ with the initial metric $g(0)=\langle\cdot,\cdot\rangle_{t_A(x^{0},1,1), x^0, 1, 1}$. We identify $g(\ell)$ with the quadruple of functions $(t(\ell),x(\ell),s(\ell),s(\ell))$ and 
consider the map $f$ defined as
\begin{equation}\label{EQ:def_f}
    f(\ell)=F\left(g(\ell),1\right).
\end{equation}

 The  differential $dF_p:T_pV\cong\R^5\to T_1\R\cong \R$, where $V$ denotes the domain $\R_+\times D_{\sigma}\times \textstyle{\left(0, 1\right]}$,  is related to the derivative $f^\prime(0)$ as follows:
\begin{equation*}
dF_p (g^\prime(0),0)=f^\prime(0),
\end{equation*}
where $g^\prime(0)=(t^\prime(0),x^\prime(0),s^\prime(0),s^\prime(0))$. Now note that:
$$
f'(0)= \frac{d}{d\ell}\left(F\left(g(\ell),1\right)\right)\Bigr|_{\ell=0}=
\frac{d}{d\ell}\left(\frac{t_A(x(\ell),s(\ell),s(\ell),1)}{t(\ell)}\right)\Bigr|_{\ell=0}.
$$
Recall that the right-hand side is negative, as we have shown in Equation~\eqref{EQ:derivative}. We conclude that 
\begin{equation}\label{EQ:derivative_new_notation}
dF_p (g^\prime(0),0)<0.
\end{equation}

Now we are ready to state the technical result which implies Theorem~\ref{THM:main_thm_AW}.

\begin{theorem}\label{THM:convergence}
Let $\{\xi_n\}_{n\in\mathbb N}$ be any sequence converging to $1$ with $\xi_n\in \mathbb Q\cap  \left(0, 1\right]$. Then there is some $n_0$ such that all Aloff--Wallach spaces $W_{\xi_n}$ with $n\geq n_0$ have a $\sec >0$ metric which evolves under the Ricci flow to a metric with non-positively curved $2$-planes.    
\end{theorem}

 We warn the reader that the sequence $\{\xi_n\}_{n\in\mathbb N}$ is not related to the functions $\xi_0,\xi_1,\xi_2$ appearing in the proof of Lemma~\ref{LEM:omega_open}.
 
\begin{proof}
For each sequence $\{\xi_n\}_{n\in\mathbb N}$ as in the statement we define 
$$
p_n=(t_A( x^0,1,1,\xi_n),x^0,1,1,\xi_n),
$$
where $x^0\in (0.8,1)$ has been fixed in the definition of $p$, see Equation~\eqref{EQ:def_of_p}. Recall from Lemma~\ref{LEM:omega_open} that $(x^0,1,1)\in \Omega_{\hat\lambda}^\prime(\xi_n)$ for all $n$. Note that, because $t_A$ is smooth, the sequence $\{p_n\}_{n\in\mathbb N}$ converges to $p$.

We think of the first four coordinates of $p_n$ and $p$ as the corresponding metric $\langle\cdot,\cdot\rangle_{t,\vec s}$ on $W_{\xi_n}^7$ and on $W_1^7$, respectively.  We denote their Ricci flows by $g_{n}(\ell)=\langle\cdot,\cdot\rangle_{t^{n}(\ell), s^{n}_{0}(\ell), s_{1}^{n}(\ell), s_{2}^{n}(\ell)}$ and $g(\ell)=\langle\cdot,\cdot\rangle_{t(\ell), x(\ell), s(\ell), s(\ell)}$, respectively. They are determined by the following systems, where we drop $\ell$ from the notation for brevity:
\begin{align}\label{Eq:Ricci_Flow_sequence}
&\begin{cases}
(t^n)^\prime=-2r_0(t^n,s_0^n,s_1^n,s_2^n,\xi_n) t^n,\\
(s_j^n)'=-2r_{j+1}(t^n,s_0^n,s_1^n,s_2^n,\xi_n) s_j^n, \text{ for all } j\in\{0,1,2\},\\
(t^n(0),s_0^n(0),s_1^n(0),s_2^n(0),\xi_n)=p_n.
\end{cases}\\ \notag\\
\label{Eq:Ricci_Flow_limit}
&\begin{cases}
t^\prime=-2r_0(t,s_0,s_1,s_2,1) t,\\
s_j'=-2r_{j+1}(t,s_0,s_1,s_2,1) s_j, \text{ for all } j\in\{0,1,2\},\\
(t(0),s_0(0),s_1(0),s_2(0),\xi)=p.
\end{cases}
\end{align}
These systems have a unique solution in an open interval $(-L^n,L^n)$ and $(-L,L)$ for some $L^n>0$ and $L>0$, respectively. 

Since $r_{i}$'s are smooth functions on $\R_+^4\times(0,1]$ and $p_n$ converges to $p$, the right-hand side of System~\eqref{Eq:Ricci_Flow_sequence} converges to the right-hand side of System~\eqref{Eq:Ricci_Flow_limit}. Consequently, the same holds for the left-hand sides, so $g_{n}'(0)$ converges to $g'(0)$. 

Denote $f_n(\ell)=F(g_n(\ell),\xi_n)$, where $F$ is the map from Equation~\eqref{EQ:map_F}. Then $f_n^\prime(0)=dF_{p_n}(g_n^\prime (0),0)$. Due to all the convergences that have been discussed above, it follows that $f_n^\prime(0)$ converges to $f^\prime(0)$, where $f$ is the map in Equation~\eqref{EQ:def_f}. Since $f^\prime(0)<0$, we conclude that there is some $n_0$ such that for all $n\geq n_0$ it holds that
$$
f_n^\prime(0)<0.
$$
It follows that for each $n\geq n_0$ there is a positive time $L_1^n\leq L^n$ such that $f_n(\ell)>1$ for $\ell\in (-L_1^n,0)$ and $f_n(\ell)<1$ for $\ell\in (0,L_1^n)$. Equivalently, the following properties hold for any $n\geq n_0$:
\begin{itemize}
\item[(P1)] $t_A(s_0^n(\ell),s_1^n(\ell),s_2^n(\ell), \xi_n)>t^n(\ell)$ for all $\ell\in (-L_1^n,0)$,
\item[(P2)] $t_A(s_0^n(\ell),s_1^n(\ell),s_2^n(\ell), \xi_n)<t^n(\ell)$ for all $\ell\in (0,L_1^n)$.
\end{itemize}

Note also that, because $(x^0,1,1)\in \Omega_{\hat\lambda}^\prime(\xi_n)$ for all $n$ and $\Omega_{\hat\lambda}^\prime(\xi_n)$ is an open set (see Lemma~\ref{LEM:omega_open}), there exists some positive $L_2^n\leq L^n$ such that:
\begin{itemize}
\item[(P3)] $(s_0^n(\ell),s_1^n(\ell),s_2^n(\ell))\in\Omega_{\hat\lambda}^\prime(\xi_n)$ for all $\ell\in (-L_2^n,L_2^n)$.
\end{itemize}

By defining $\bar L^n=\min\{L^n_1,L^n_2\}$, it follows that for any $n\geq n_0$, (P1), (P2) and (P3) hold for all $\ell\in (-\bar L^n,\bar L^n)$. By Theorem~\ref{THM:Omega}, the metrics $g_n(\ell)$  on $W^{7}_{\xi_n}$ are positively curved for $\ell\in (-\bar L^n,0)$ and have non-positively curved $2$-planes for $\ell\in (0,\bar L^n)$. In summary, any metric $g_n(\ell)$ with $\ell\in (-\bar L^n,0)$ has $\sec>0$ and evolves to metrics with non-positively curved $2$-planes. 
\end{proof}

\subsection{General Aloff--Wallach Spaces}\label{S:computer_assisted}

In this subsection we prove that all Aloff--Wallach spaces $W^7_{\xi}$ with $\xi\in(0, 1]\cap \mathbb Q$ have $\sec>0$ metrics which evolve under the Ricci flow to metrics with non-positively curved $2$-planes. We use the notation from Section~\ref{S:Generic_Case}. Some computations are done with Maple, due to their complexity.

For each $\xi\in(0, 1]\cap \mathbb Q$ and each $x\in (0,1)$, let $g$ denote the metric $\langle\cdot,\cdot\rangle_{t_A(x,1,1,\xi),x,1,1}$ on $W_\xi$, and let  $g_\xi(\ell)=\langle\cdot,\cdot\rangle_{t(\ell),s_0(\ell),s_1(\ell),s_2(\ell)}$ be the Ricci flow with initial condition $g_\xi(0)=g$, which exists in some interval $(-L_\xi,L_\xi)$. Although $g$ depends on $x$ and $\xi$ (and $g_\xi(\ell)$ on $x$), we do not explicitly denote it, since there is no risk of confusion.

We denote by $F_\xi:\R_+\times D_\sigma\to\R$ the map $F$ from Equation~\eqref{EQ:map_F} restricted to $\R_+\times D_\sigma\times\{\xi\}$. Consider the smooth map 
$$
f_{\xi}(\ell)=F_\xi(g_\xi(\ell)).
$$
We are going to show that for each $\xi\in(0, 1]\cap \mathbb Q$ there is some $x\in (0,1)$ for which $f_{\xi}^\prime(0)<0$. This, as explained in the proof of Theorem~\ref{THM:convergence}, implies that there is some positive $\bar L_\xi\leq L_\xi$ such that the metric $g_\xi(\ell)$  on $W^{7}_{\xi}$ is positively curved for $\ell\in (-\bar L_\xi,0)$ and have non-positively curved $2$-planes for $\ell\in (0,\bar L_\xi)$. In summary, any metric $g_\xi(\ell)$ with $\ell\in (-\bar L_\xi,0)$ has $\sec>0$ and evolves to metrics with non-positively curved $2$-planes. 

Note that $f_{\xi}^\prime(0)=\langle \nabla_g F_\xi, g^\prime_\xi(0)\rangle$, where $\nabla_g F_\xi$ denotes the value at $g$ of the gradient
$$
\nabla F_\xi=\left(\frac{\partial F_\xi}{\partial t},  \frac{\partial F_\xi}{\partial s_0}, \frac{\partial F_\xi}{\partial s_1} , \frac{\partial F_\xi}{\partial s_2}\right).
$$
The computation of $\frac{\partial F_\xi}{\partial s_i}$ involves the derivative of $t_A$ and hence of $\tilde A^{-1}$ with respect to $s_i$. As explained in Section~\ref{SS:special_spaces_metrics}, the expression of $\tilde A^{-1}$ for arbitrary triples $\vec s$ is complicated. In order to simplify the computations as much as possible, we use the fact that 
\begin{equation}\label{eq:Inverse_Matrix_Der}
\frac{\partial \tilde A^{-1}}{\partial s_i}(x, 1,1)=-\tilde A^{-1}(x, 1, 1)\frac{\partial \tilde A}{\partial s_i}(x, 1,1)\tilde A^{-1}(x, 1, 1). 
\end{equation} 
This formula easily follows from the observation that the derivative of $I=\tilde A\tilde A^{-1}$ vanishes. The upshot of Equation~\eqref{eq:Inverse_Matrix_Der} is that we only have to compute $\tilde A^{-1}$ at $(x,1,1)$, which we have already done in Section~\ref{SS:3_param}, and that we have the expression for $\tilde A$ at hand, so we can compute $\frac{\partial \tilde A}{\partial s_i}$ directly and then evaluate at $(x,1,1)$.

Using Equation~\eqref{eq:Inverse_Matrix_Der} and the fact that $\tilde A^{-1}$ is a symmetric matrix, we find that
\begin{align}\label{EQ:Partial_Der}
\frac{\partial F_\xi}{\partial s_i}(g)
    =-\frac{2}{3x(4-x)}\left(\frac{2\langle \frac{\partial v}{\partial s_i}, \tilde A^{-1}v_\xi\rangle-\langle \tilde A^{-1}v_\xi, \frac{\partial \tilde A}{\partial s_i}(\tilde A^{-1}v_\xi)\rangle}{\langle v_\xi, \tilde A^{-1}v_\xi\rangle^2}\right),
    \end{align}
where $\frac{\partial v}{\partial s_i}$ is evaluated at $(x,1,1,\xi)$, the matrices $\tilde A^{-1}$ and $\frac{\partial \tilde A}{\partial s_i}$ are evaluated at $(x,1,1)$, and $v_\xi$ denotes $v(x,1,1,\xi)$.

Let us compute the terms appearing in Equation~\eqref{EQ:Partial_Der}. Along the way we will indicate which computations have been done with Maple. First we introduce some notations to simplify the expressions. Let
$$P_i:= \langle \frac{\partial v}{\partial s_i}, \tilde A^{-1}v_\xi\rangle,\qquad Q_i:=\langle \tilde A^{-1}v_\xi, \frac{\partial \tilde A}{\partial s_i}(\tilde A^{-1}v_\xi)\rangle,\qquad L=\langle v_\xi, \tilde A^{-1}v_\xi\rangle.$$ 

Then Equation~\eqref{EQ:Partial_Der} can be written as: 
\begin{align}\label{EQ:Partial_Der_2}
    \frac{\partial F_\xi}{\partial s_i}(g)
    &=-\frac{2}{3x(4-x)}\left(\frac{2P_i-Q_i}{L^2}\right).
    \end{align}

First we compute the term which is involved in all of $L,P_i,Q_i$:
    \begin{equation}\label{EQ:W}
       W:=\tilde A^{-1}v_\xi
       =\frac{1}{( x^{2}-5 x +4)\sqrt{2(\xi^2+\xi+1)  } }\begin{pmatrix}
\frac{(x -2) (\xi +1)}{2} \\
 \frac{(\xi -1) x^{2}+(-5 \xi +5) x -2 \xi -10}{12} \\
-\frac{(\xi -1) x^{2}+(-5 \xi +5) x +10 \xi +2}{12} \\
\end{pmatrix}.
\end{equation}
We use it to obtain
\begin{equation*}\label{EQ:L}
L=\langle v_\xi,W\rangle=\frac{x^{2} (\xi-1)^{2}-4 x (\xi-1)^2-12(\xi+1)^2}{24 x  (x-4) (\xi^{2}+\xi+1)}.
\end{equation*}
Next, we differentiate  $v$ with respect to $s_i$ and evaluate at $(x, 1, 1)$:   
   \begin{align*}
       \frac{\partial v}{\partial s_0}(x, 1,1)=\left(\frac{1+\xi}{x^2\sqrt{2(\xi^2+\xi+1)}}, 0, 0\right), 
       \\
\frac{\partial v}{\partial s_1}(x, 1,1)=\left( 0,-\frac{\xi}{\sqrt{2(\xi^2+\xi+1)}}, 0\right),
\\
\frac{\partial v}{\partial s_2}(x, 1,1)=\left( 0, 0, -\frac{1}{\sqrt{2(\xi^2+\xi+1)}}\right).
   \end{align*}     
We can now compute $P_i=\langle \frac{\partial v}{\partial s_i}, W\rangle$: 
   \begin{align*}
      P_0&=\frac{(\xi+1)^{2} (x-2)}{4 (\xi^{2}+\xi+1) x^{2} (x-1) (x-4)},  
      \\[7pt]
 P_1&=-\frac{((\xi -1) x^{2}+(-5 \xi +5) x -2 \xi -10) \xi}{24 (\xi^{2}+\xi+1) (x-1) (x-4)},  
 \\[7pt]
 P_2&=\frac{(\xi-1) x^{2}+(-5\xi+5) x+10 \xi+2}{24 (\xi^{2}+\xi+1) (x-1) (x-4)}. 
  \end{align*}
Next, we differentiate $\tilde A$ with respect to $s_i$ and evaluate at $(x, 1, 1)$:
 \begin{align*}
\frac{\partial \tilde A}{\partial s_0}(x, 1,1)&=\begin{pmatrix}
-\frac{4}{x^{2}} & 1 & 1 
\\
 1 & 0 & 0 
\\
 1 & 0 & 0 
\end{pmatrix},\\ 
\frac{\partial \tilde A}{\partial s_1}(x, 1,1)&=\begin{pmatrix}
0 & \frac{-x^{2}+x+1}{x} & -\frac{(x-1)^{2}}{x} 
\\
 \frac{-x^{2}+x+1}{x} & -4 & 1 
\\
 -\frac{(x-1)^{2}}{x} & 1 & 0 
\end{pmatrix},\\
\frac{\partial \tilde A}{\partial s_2}(x, 1,1)&=
\begin{pmatrix}
0   
& -\frac{(x-1)^{2}}{x} & \frac{-x^{2}+x+1}{x}
\\
 -\frac{(x-1)^{2}}{x} & 0 & 1 
\\
 \frac{-x^{2}+x+1}{x} & 1 & -4 \end{pmatrix}.
\end{align*}

Let us denote $U_i:=\frac{\partial \tilde A}{\partial s_{i}}(W)$. Then we have
\begin{align*}
U_0&=\frac{1}{( x^{2}-5 x +4)\sqrt{2(\xi^2+\xi+1)  } }\begin{pmatrix}
-\frac{(x^{2}+2 x-4) (\xi+1)}{x^2} 
\\
 \frac{(x-2) (\xi+1)}{2} 
\\
 \frac{(x-2) (\xi+1)}{2} 
\end{pmatrix}, 
\\
U_1&=\frac{1}{( x^{2}-5 x +4)\sqrt{2(\xi^2+\xi+1)  } }\begin{pmatrix}
-\frac{(\xi-1) x^{3}+(-19 \xi-5) x^{2}+(32 \xi+4) x-8 \xi+8}{12x} 
\\
 \frac{(-11 \xi-1) x^{3}+(43 \xi-7) x^{2}+(-8 \xi+32) x-12 \xi-12}{12x} 
\\
 \frac{(-5 \xi-7) x^{3}+(19 \xi+29) x^{2}+(-32 \xi-40) x+12 \xi+12}{12x} 
\end{pmatrix}, 
\\
U_2&=\frac{1}{( x^{2}-5 x +4)\sqrt{2(\xi^2+\xi+1)  } }\begin{pmatrix}
\frac{(\xi-1) x^{3}+(5 \xi+19) x^{2}+(-4 \xi-32) x-8 \xi+8}{12x} 
\\
 \frac{(-7 \xi-5) x^{3}+(29 \xi+19) x^{2}+(-40 \xi-32) x+12 \xi+12}{12x} 
\\
 -\frac{(\xi+11) x^{3}+(7 \xi-43) x^{2}+(-32 \xi+8) x+12 \xi+12}{12x} 
\end{pmatrix}. 
  \end{align*}

  Lastly, compute $Q_i=  \langle W, U_i\rangle$: 
  \begin{align}
 Q_0=& -\frac{(x+2) (\xi+1)^{2} (x-2)}{2 x^{2} (x-4)^{2} (x-1) (\xi^{2}+\xi+1)},\nonumber 
 \\[5pt]
 Q_1=&\frac{-(\xi-1)^{2} x^{4}+(7 \xi^{2}-18 \xi+11) x^{3}+(24 \xi^{2}+96 \xi-24) x^{2}}{48 x (x-4)^{2} (x-1) (\xi^{2}+\xi+1)} \label{EQ:Q_1}\\[5pt]
  \nonumber
&+ \frac{(-116 \xi^{2}-152 \xi+28) x+32 \xi^{2}-32}{48 x (x-4)^{2} (x-1) (\xi^{2}+\xi+1)}, \\[5pt]
   Q_2=&\frac{-(\xi-1)^{2} x^{4}+(11 \xi^{2}-18 \xi+7) x^{3}+(-24 \xi^{2}+96 \xi+24) x^{2}}{48 x (x-4)^{2} (x-1) (\xi^{2}+\xi+1)} \label{EQ:Q_2}
   \\[5pt]  \nonumber
   &+\frac{(28 \xi^{2}-152 \xi-116) x-32 \xi^{2}+32}{48 x (x-4)^{2} (x-1) (\xi^{2}+\xi+1)}. 
  \end{align}
Equations~\eqref{EQ:Q_1} and \eqref{EQ:Q_2} have been computed with Maple.

Having computed $L$, $P_i$ and $Q_i$, we use Equation~\eqref{EQ:Partial_Der_2} to compute the partial derivatives of $F_\xi$.

\begin{align}
    \frac{\partial F_\xi}{\partial t}(g) =&\frac{-48 \xi^{2}-48 \xi-48}{x(x-4) ((\xi-1)^{2} x^{2}-4 (\xi-1)^{2} x-12 (\xi+1)^{2}) },\nonumber \\[6pt]
    \frac{\partial F_\xi}{\partial s_0}(g) =& \frac{384 (\xi+1)^{2} (x-2) (\xi^{2}+\xi+1)}{x(x-4) ((\xi-1)^{2} x^{2}-4 (\xi-1)^{2} x-12 (\xi+1)^{2})^{2} }, \label{EQ:der_F_s0}\\[6pt]
    \frac{\partial F_\xi}{\partial s_1}(g)=&\frac{-24 (\xi^{2}+\xi+1) ((\xi^{2}-\frac{2}{3} \xi-\frac{1}{3}) x^{4}+(-\frac{29}{3} \xi^{2}+6 \xi+\frac{11}{3}) x^{3}}{(x-4) (x-1) ((\xi-1)^{2} x^{2}-4 (\xi-1)^{2} x-12 (\xi+1)^{2})^{2}}\label{EQ:der_F_s1}\\[6pt]  \nonumber 
&+\frac{(32 \xi^{2}-8 \xi-8) x^{2}+(-28 \xi^{2}+\frac{8}{3} \xi+\frac{28}{3})x+\frac{32 \xi^{2}}{3}-\frac{32}{3})}{(x-4) (x-1) ((\xi-1)^{2} x^{2}-4 (\xi-1)^{2} x-12 (\xi+1)^{2})^{2}}, \\[6pt]
\frac{\partial F_\xi}{\partial s_2}(g)=& \frac{8 (\xi^{2}+\xi +1) ((\xi^{2}+2 \xi -3) x^{4}+(-11 \xi^{2}-18 \xi +29) x^{3}}{(x -4) (x -1) ((\xi -1)^{2} x^{2}-4 (\xi -1)^{2} x -12 (\xi +1)^{2})^{2}} \label{EQ:der_F_s2}\\[6pt] \nonumber 
&+\frac{(24 \xi^{2}+24 \xi -96) x^{2}+(-28 \xi^{2}-8 \xi +84) x +32 \xi^{2}-32)}{(x -4) (x -1) ((\xi -1)^{2} x^{2}-4 (\xi -1)^{2} x -12 (\xi +1)^{2})^{2}}.
\end{align}

Equations~\eqref{EQ:der_F_s0}, \eqref{EQ:der_F_s1} and \eqref{EQ:der_F_s2} have been computed with Maple.

The last piece of information we need to compute $f_\xi^\prime(0)$ is the Ricci flow equation at time $0$, that is,   $g^\prime_\xi(0)=(t^\prime(0),s_0^\prime(0), s_1^\prime(0), s_2^\prime(0) )$, which is given by
\begin{align*}
    t^\prime(0)&=-\frac{(x^{4}-8 x^{3}+16 x^{2}) (\frac{(\xi+1)^{2}}{x^{2}}+\xi^{2}+1)}{3 (\xi^{2}+ \xi+1)}, \\
    s_0^\prime(0)&=
 -8+\frac{(\xi+1)^{2} (4-x)}{\xi^{2}+\xi+1}-2 x^{2}, \\
 s_1^\prime(0)&=-12+\frac{\xi^{2} (-x^{2}+4 x)}{\xi^{2}+\xi+1}+2 x,\\
 s_2^\prime(0)&=-12+\frac{-x^{2}+4 x}{\xi^{2}+\xi+1}+2 x.
\end{align*}

At this stage we have all the materials to compute $ f_{\xi}^\prime(0)=\langle \nabla_g F_\xi, g^\prime_\xi(0)\rangle$. Using Maple we obtain:

\begin{equation}\label{EQ:DER}
  f_{\xi}^\prime(0)=\frac{K(\xi, x)}{x(x-4) (x-1) ((\xi-1)^{2} x^{2}-4 (\xi-1)^{2} x-12 (\xi+1)^{2})^{2}},
\end{equation}

where 
\begingroup
\addtolength{\jot}{1em}
\begin{align*}
  K(\xi, x)=&40 (\xi-1)^{2} \left(\xi^{2}+\frac{4}{5} \xi+1\right) x^{7}-568 (\xi-1)^{2} \left(\xi^{2}+\frac{60}{71} \xi+1\right) x^{6}\\\nonumber
    &+{(2960 \xi^{4}-3552 \xi^{3}+416 \xi^{2}-3552 \xi+2960) x^{5}}\\\nonumber
    &+{(-7664 \xi^{4}+6912 \xi^{3}-2336 \xi^{2}+6912 \xi-7664) x^{4}}\\\nonumber
    &+{(9856 \xi^{4}-3968 \xi^{3}+5120 \xi^{2}-3968 \xi+9856) x^{3}}\\\nonumber
    &+{(-5312 \xi^{4}+6144 \xi^{3}+10624 \xi^{2}+6144 \xi-5312) x^{2}}\\\nonumber
    &+256 (\xi^{2}-50 \xi+1) (\xi+1)^{2} x+6144 \xi (\xi+1)^{2}.
\end{align*}
\endgroup

 Let us analyze the sign of the quotient in Equation~\eqref{EQ:DER}. First, note that the denominator is positive for any $x<1$. As for the numerator, we compute that the left-sided limit 
$$
\lim_{x\to 1^-} K(\xi, x)=-432(\xi^2-1)^2
$$
is negative for all $\xi\in (0,1)$. It follows that for each $\xi\in (0,1)$ there is some $x=x(\xi)<1$ sufficiently close to $1$ such that $K(\xi, x)<0$. Hence, we have found $x$ for which $f_{\xi}^\prime(0)<0$, as desired.

In the case $\xi=1$, numerator and denominator have $768(x-1)$ as a common factor. When simplifying, one arrives to the expression 
$$
f_1^\prime(0)=\frac{x^4+6x^3-16x^2-32x+32}{3x(x-4)}.
$$
Of course, this coincides with the expression in Equation~\eqref{EQ:derivative_f} for $f_1^\prime(0)$ obtained by hand in the proof of Theorem~\ref{THM:3-parameter}, where we have shown that $f_1^\prime(0)<0$ if $x\in(0.8,1)$.

\section{The Berger space}\label{SEC:Berger}

Recall that the Berger space $B^{13}$ is the homogeneous space $\su{5}/\sp{2}\mathsf{S}^1$. The subgroup $\sp{2}\mathsf{S}^1$ is defined as the image of $\sp{2}\times\mathsf{S}^1$ under the homomorphism
$$
\iota:\su{4}\times\mathsf{S}^1 \to\u{4}\subset\su{5},\qquad \iota(A,\lambda)=\lambda A.
$$
Endow $\su{5}$ with the bi-invariant metric $\llangle X,Y\rrangle =-\frac{1}{2}\text{Re}\tr (XY)$, for $X,Y\in\g{su}_5$. Let $\lieP$ be the orthogonal complement of $\g{sp}_2\oplus\R$, the Lie algebra of $\sp{2}\mathsf{S}^1$. The space $\lieP$ decomposes as $\lieP_1\oplus\lieP_2$, where $\lieP_1$ denotes the orthogonal complement of $\g{sp}_2$ in $\g{su}_4$ and $\lieP_2$ is the orthogonal complement of $\g{su}_4\oplus\R$ in $\g{su}_5$. As noted in \cite[Lemma~7.9]{Pue99}, all $\su{5}$-invariant metrics are of the form
$$
\langle \cdot, \cdot\rangle_{x_1,x_2}=x_1\llangle\cdot,\cdot\rrangle|_{\lieP_1}+ x_2\llangle\cdot,\cdot\rrangle|_{\lieP_2}.
$$
P\"uttmann computed in \cite[Theorem~7.17]{Pue99} the extremal curvatures $\langle \cdot, \cdot\rangle_{x_1,x_2}$, see also \cite[Corollary~7.18]{Pue99}. After doing the corresponding identifications between P\"uttmann's notation for the metrics (see the bottom of \cite[p.~677]{Pue99}) and ours we arrive to the following theorem.

\begin{theorem}[P\"uttmann]\label{THM:B13_cone_sec}
The metric $\langle \cdot, \cdot\rangle_{x_1,x_2}$ on $B^{13}$ has $\sec>0$ if and only if $x_1<2x_2$. Moreover, the metric $\langle \cdot, \cdot\rangle_{x_1,x_2}$ has negatively curved $2$-planes if $x_1>2x_2$.
\end{theorem}

Now we discuss the Ricci tensor of our metrics. By dimensional reasons, $\lieP_1$ and $\lieP_2$ are inequivalent as $\Ad(\sp{2}\mathsf{S}^1)$-representations. It follows that the Ricci tensor of $\langle \cdot, \cdot\rangle_{x_1,x_2}$ has a diagonal form 
\begin{equation*}
\Ric=x_1r_1 \llangle\cdot,\cdot\rrangle|_{ \lieP_1}+ x_2r_2\llangle\cdot,\cdot\rrangle|_{\lieP_2}.
\end{equation*}
The eigenvalues $r_1,r_2$ can be computed using Equation~\eqref{eqn:Ric_eig}. This has been done by Rodionov in \cite[Corollary~2.1]{Ro92}, and for the metrics $\langle \cdot, \cdot\rangle_{x_1,1}$ with $x_2=1$ he obtained:
\begin{equation}\label{Eq:B13-Ricci}
   r_1=\frac{8+x_1^2}{x_1},\qquad r_2=\frac{5(8-x_1)}{4}. 
\end{equation}

We are ready to prove Theorem~\ref{THM:main_thm_B13} for $B^{13}$. 

\begin{proof}[Proof of Theorem~\ref{THM:main_thm_B13} for $B^{13}$]
Let $g(\ell)=\langle \cdot, \cdot\rangle_{x_1(\ell),x_2(\ell)}$ the Ricci flow of the initial metric $g(0)=\langle \cdot, \cdot\rangle_{2,1}$, which is determined by the system
\begin{equation}\label{EQ:flow_B13}
\begin{cases}
x_1^\prime(\ell)=-2r_1x_1(\ell),\\
x_2^\prime(\ell)=-2r_2x_2(\ell),\\
(x_1(0),x_2(0))=(2,1).
\end{cases}
\end{equation}
This system has a unique solution in an open interval $(-L,L)$ for some $L>0$. We claim that 
\begin{equation}\label{EQ:derivative_B13}
\frac{d}{d\ell}\left(\frac{x_1(\ell)}{2x_2(\ell)}\right)\Bigr|_{\ell=0}>0.
\end{equation}
This implies that there exists $\bar L>0$ with $\bar L\leq L$ such that $x_1(\ell)<2x_2(\ell)$ for all $\ell\in (-\bar L,0)$ and $x_1(\ell)>2x_2(\ell)$ for all $\ell\in (0,\bar L)$. By Theorem~\ref{THM:B13_cone_sec}, the metrics $g(\ell)$ are positively curved for $\ell\in (-\bar L,0)$ and have negatively curved $2$-planes for $\ell\in (0,\bar L)$. In summary, any metric $g(\ell)$ with $\ell\in (-\bar L,0)$ has $\sec>0$ and evolves to metrics with negatively curved $2$-planes.

It remains to prove the claim in Equation~\eqref{EQ:derivative_B13}. Dropping $\ell$ from the equations we have
$$
\frac{d}{d\ell}\left(\frac{x_1}{2x_2}\right)=\frac{x_1^\prime x_2-x_1 x_2^\prime}{2x_2^2}=\frac{-2r_1x_1x_2+2r_2x_1x_2}{2x_2^2}{=\frac{x_1}{x_2}(r_2-r_1)},
$$
where we have used the equations in System~\eqref{EQ:flow_B13}. At $\ell=0$, we have $(x_1(0),x_2(0))=(2,1)$, so $r_1=6$ and $r_2=\frac{15}{2}$ by Equation~\eqref{Eq:B13-Ricci}. It follows that the derivative equals 3.
\end{proof}

\begin{remark}\label{Rem:Einstein_Ricci}
    Note that Theorem~\ref{THM:main_thm_B13} can be alternatively proven via the fact that both $W^7_{1, 1}$  and $B^{13}$  admit precisely two Einstein metrics, one with positive sectional curvature and one with $2$-planes with negative sectional curvature,  see \cite[Lemma~2]{N04} for $W^7_{1,1}$  and \cite[Theorem~3.1]{Ro92} for $B^{13}$ (see also Theorems~\ref{THM:cone_U2_invariant} and \ref{THM:B13_cone_sec} for the sign of sectional curvatures).

    Consider the normalized Ricci flow  of the two-parameter  metrics $\langle\cdot,\cdot\rangle_{t, t, s, s}$  on $W^{7}_{1, 1}$ and the metrics $\langle\cdot,\cdot\rangle_{x_1, x_2}$ on $B^{13}$.  In both cases the normalized Ricci flow is then given by one equation in one variable. Then there exists a normalized Ricci flow solution $g(\ell)$, defined on the maximal time interval $(L_{\min}, L_{\max})$  with the property that 
    $$\lim_{\ell \to L_{\min}}g(\ell)=g_1,$$
    $$\lim_{\ell\to L_{\max}}g(\ell)=g_2,$$   
     where $g_1$ is  the Einstein metric on $W^7_{1,1}$ or $B^{13}$ with $\sec>0$ and $g_2$ is the Einstein metric which has $2$-planes with negative curvature.

\end{remark}

\end{document}